\documentclass[11pt, reqno]{amsart}
\usepackage{amsmath, amssymb}
\usepackage{color}
\usepackage{bm, mathrsfs}
\usepackage[normalem]{ulem}
\usepackage{graphicx}
\usepackage{tikz}
\usepackage[all]{xy}
\allowdisplaybreaks 

\newtheorem{theorem}{Theorem}[section]
\newtheorem{lemma}[theorem]{Lemma}
\newtheorem{proposition}[theorem]{Proposition}
\newtheorem{corollary}[theorem]{Corollary}

\theoremstyle{definition}
\newtheorem{definition}[theorem]{Definition}
\newtheorem{example}[theorem]{Example}

\theoremstyle{remark}
\newtheorem{remark}[theorem]{Remark}
\numberwithin{equation}{section}

\newcommand{\Z}{\mathbb{Z}}
\newcommand{\Q}{\mathbb{Q}}

\newcommand{\C}{\mathcal{C}}

\newcommand{\I}{\mathcal{I}}

\newcommand{\K}{\mathcal{K}}
\newcommand{\M}{\mathcal{M}}
\newcommand{\cc}{\mathfrak{c}}

\renewcommand{\ss}{\mathfrak{s}}
\renewcommand{\S}{\mathfrak{S}}

\newcommand{\A}{\mathcal{A}}

\newcommand{\F}{\mathcal{F}}

\newcommand{\Ztilde}{\widetilde{Z}}
\newcommand{\zz}{\bar{\bar{z}}}
\newcommand{\ZZ}{\overline{\overline{Z}}}

\renewcommand{\sl}{\mathfrak{sl}}

\newcommand{\bu}{\mathfrak{bu}}
\newcommand{\Lk}{\operatorname{Lk}}
\newcommand{\id}{\mathrm{id}}
\newcommand{\Id}{\mathrm{Id}}
\newcommand{\pr}{\mathrm{pr}}

\newcommand{\Sp}{\mathrm{Sp}}

\newcommand{\LCob}{\mathcal{LC}\mathit{ob}}
\newcommand{\tsA}{{}^\mathit{ts}\!\!\mathcal{A}}
\newcommand{\Mag}{\mathrm{Mag}}

\newcommand{\rank}{\operatorname{rank}}
\renewcommand{\Im}{\operatorname{Im}}
\newcommand{\Ker}{\operatorname{Ker}}

\newcommand{\tor}{\operatorname{tor}}

\newcommand{\ideg}{\operatorname{i-deg}}

\newcommand{\ang}[1]{\langle #1 \rangle}

\newcommand{\Ygraph}[4]{%
\begin{tikzpicture}[scale=#1, baseline={(0,-0.1)}, densely dashed]
  \coordinate (origin) at (0,0);
  \draw (origin) -- (-1,1) node[at end, anchor=south] {\Small $#2$};
  \draw (origin) -- (1,1) node[at end, anchor=south] {\Small $#3$};
  \draw (origin) -- (0,-1) node[at end, anchor=north] {\Small $#4$};
\end{tikzpicture}%
}

\newcommand{\Ygraphref}[4]{%
\begin{tikzpicture}[scale=#1, baseline={(0,-0.1)}, densely dashed]
  \coordinate (origin) at (0,0);
  \draw (origin) -- (-1,-1) node[at end, anchor=north] {\Small $#3$};
  \draw (origin) -- (1,-1) node[at end, anchor=north] {\Small $#4$};
  \draw (origin) -- (0,1) node[at end, anchor=south] {\Small $#2$};
\end{tikzpicture}%
}

\newcommand{\lambdagraph}[4]{%
  \begin{tikzpicture}[scale=#1, baseline={(0,-0.1)}, densely dashed]
    \coordinate (origin) at (0,0);
    \draw (origin) -- (-1,-1) node[at end, anchor=north] {\Small $#3$};
    \draw (origin) -- (1,-1) node[at end, anchor=north] {\Small $#4$};
    \draw (origin) -- (0,1) node[at end, anchor=south] {\Small $#2$};
  \end{tikzpicture}%
}

\newcommand{\buYref}[4]{%
  \begin{tikzpicture}[scale=#1, baseline={(0,-0.1)}, densely dashed]
    \coordinate (origin) at (0,0);
    \draw (origin) -- (210:2) node[at end, anchor=north] {\Small $#3$} ;
    \draw (origin) -- (-30:2) node[at end, anchor=north] {\Small $#4$};
    \draw (origin) -- (90:2) node[at end, anchor=south] {\Small $#2$};
    \draw[fill=white] (origin) circle [radius=0.8];
  \end{tikzpicture}%
}

\newcommand{\Hgraph}[4]{%
\begin{tikzpicture}[scale=0.4, baseline={(0,-0.1)}, densely dashed]
  \draw (-1,1) -- (-1,-1) node[at start, anchor=south] {\Small $#1$} node[at end, anchor=north] {\Small $#3$};
  \draw (1,1) -- (1,-1) node[at start, anchor=south] {\Small $#2$} node[at end, anchor=north] {\Small $#4$};
  \draw (-1,0) -- (1,0) ;
\end{tikzpicture}%
}

\newcommand{\YYgraph}[4]{%
\begin{tikzpicture}[scale=0.3, baseline={(0,0.6)}, densely dashed]
  \draw (1,1) -- (4,4) node[at start, anchor=north] {\Small $#4$} node[at end, anchor=south] {\Small $#3$};
  \draw (2,2) -- (0,4) node[anchor=south] {\Small $#1$};
  \draw (3,3) -- (2,4) node[anchor=south] {\Small $#2$};
\end{tikzpicture}%
}

\newcommand{\YYgraphref}[4]{%
\begin{tikzpicture}[scale=0.3, baseline={(0,0.6)}, densely dashed]
  \draw (3,1) -- (0,4) node[at start, anchor=north] {\Small $#4$} node[at end, anchor=south] {\Small $#1$};
  \draw (2,2) -- (4,4) node[anchor=south] {\Small $#3$};
  \draw (1,3) -- (2,4) node[anchor=south] {\Small $#2$};
\end{tikzpicture}%
}

\newcommand{\lamlam}[4]{%
\begin{tikzpicture}[scale=0.3, baseline={(0,0.4)}, densely dashed]
  \draw (1,3) -- (4,0) node[at start, anchor=south] {\Small $#1$} node[at end, anchor=north] {\Small $#4$};
  \draw (2,2) -- (0,0) node[anchor=north] {\Small $#2$};
  \draw (3,1) -- (2,0) node[anchor=north] {\Small $#3$};
\end{tikzpicture}%
}

\newcommand{\lamlamref}[4]{%
\begin{tikzpicture}[scale=0.3, baseline={(0,0.4)}, densely dashed]
  \draw (3,3) -- (0,0) node[at start, anchor=south] {\Small $#1$} node[at end, anchor=north] {\Small $#2$};
  \draw (1,1) -- (2,0) node[anchor=north] {\Small $#3$};
  \draw (2,2) -- (4,0) node[anchor=north] {\Small $#4$};
\end{tikzpicture}%
}

\newcommand{\dumbbell}[2]{%
\begin{tikzpicture}[scale=0.3, baseline={(0,-0.3)}, densely dashed]
  \draw (0,0) circle [radius=1];
  \draw (-1,0) .. controls +(-1,0) and +(0,1) .. (-2,-2) node[anchor=north] {\Small $#1$};
  \draw (1,0) .. controls +(1,0) and +(0,1) .. (2,-2) node[anchor=north] {\Small $#2$};
\end{tikzpicture}%
}

\newcommand{\Hbottom}[4]{%
\begin{tikzpicture}[scale=0.5, baseline={(0,0.4)}, densely dashed]
  \draw (-1,0) arc (180:0:1) node[at start, anchor=north] {\Small $#2$} node[at end, anchor=north] {\Small $#3$};
  \draw (-2,0) arc (180:0:2) node[at start, anchor=north] {\Small $#1$} node[at end, anchor=north] {\Small $#4$};
  \draw (0,1) -- (0,2);
\end{tikzpicture}%
}

\newcommand{\Ybottom}[3]{%
\begin{tikzpicture}[scale=0.5, baseline={(0,0.2)}, densely dashed]
  \draw (-1,0) arc (180:0:1) node[at start, anchor=north] {\Small $#1$} node[at end, anchor=north] {\Small $#3$};
  \draw (0,1) -- (0,0) node[anchor=north] {\Small $#2$};
\end{tikzpicture}%
}

\newcommand{\comb}[4]{%
\begin{tikzpicture}[scale=0.5, baseline={(0,0.4)}, densely dashed]
  \draw (0,0) arc (180:0:1.5) node[at start, anchor=north] {\Small $#1$} node[at end, anchor=north] {\Small $#4$};
  \draw (1,1.4) -- (1,0) node[anchor=north] {\Small $#2$};
  \draw (2,1.4) -- (2,0) node[anchor=north] {\Small $#3$};
\end{tikzpicture}%
}

\newcommand{\strutbottom}[2]{%
\begin{tikzpicture}[scale=0.3, baseline={(0,0.2)}, densely dashed]
  \draw (0,0) arc (180:0:1) node[at start, anchor=north] {\Small $#1$} node[at end, anchor=north] {\Small $#2$};
\end{tikzpicture}%
}

\newcommand{\strutgraph}[2]{%
\begin{tikzpicture}[scale=0.7, baseline={(0,0.3)}, densely dashed]
  \draw (0,1) -- (0,0) node[at start, anchor=south] {\Small $#1$} node[at end, anchor=north] {\Small $#2$};
\end{tikzpicture}%
}

\newcommand{\strutt}[2]{%
\begin{tikzpicture}[scale=0.4, baseline={(0,0.4)}, densely dashed]
  \draw (0,0) -- (0,2) node[anchor=south] {\Small $#1$};
  \draw (2,0) -- (2,2) node[anchor=south] {\Small $#2$};
\end{tikzpicture}%
}

\newcommand{\struttt}[3]{%
\begin{tikzpicture}[scale=0.4, baseline={(0,0.4)}, densely dashed]
  \draw (0,0) -- (0,2) node[anchor=south] {\Small $#1$};
  \draw (2,0) -- (2,2) node[anchor=south] {\Small $#2$};
  \draw (4,0) -- (4,2) node[anchor=south] {\Small $#3$};
\end{tikzpicture}%
}

\newcommand{\yokostrut}[1]{%
\begin{tikzpicture}[scale=0.4, baseline={(0,0)}, densely dashed]
  \draw (0,0) -- (2,0) node[at start, anchor=south] {\Small $#1$};
\end{tikzpicture}%
}

\newcommand{\Tref}[1]{%
\begin{tikzpicture}[scale=0.4, baseline={(0,0.4)}, densely dashed]
  \draw (0,0) -- (0,2) node[anchor=south] {\Small $#1$};
  \draw (-1,0) -- (1,0);
\end{tikzpicture}%
}

\newcommand{\Ytop}[3]{%
\begin{tikzpicture}[scale=0.5, baseline={(0,-0.3)}, densely dashed]
  \draw (-1,0) arc (180:360:1) node[at start, anchor=south] {\Small $#1$} node[at end, anchor=south] {\Small $#3$};
  \draw (0,-1) -- (0,0) node[anchor=south] {\Small $#2$};
\end{tikzpicture}%
}

\newcommand{\Htop}[4]{%
\begin{tikzpicture}[scale=0.3, baseline={(0,-0.4)}, densely dashed]
  \draw (2,0) arc (180:360:1) node[at start, anchor=south] {\Small $#2$} node[at end, anchor=south] {\Small $#3$};
  \draw (0,0) arc (180:360:3) node[at start, anchor=south] {\Small $#1$} node[at end, anchor=south] {\Small $#4$};
  \draw (3,-1) -- (3,-3);
\end{tikzpicture}%
}

\newcommand{\buYtop}[3]{%
  \begin{tikzpicture}[scale=0.3, baseline={(0,0.1)}, densely dashed]
    \draw (0,0) -- (-2,2) node[at end, anchor=south] {\Small $#1$} ;
    \draw (0,0) -- (0,2) node[at end, anchor=south] {\Small $#2$};
    \draw (0,0) -- (2,2) node[at end, anchor=south] {\Small $#3$};
    \draw[fill=white] (0,0) circle [radius=1];
  \end{tikzpicture}%
}

\newcommand{\HHtop}[5]{%
\begin{tikzpicture}[scale=0.5, baseline={(0,-0.5)}, densely dashed]
  \draw (0,-1) -- (0,0) node[anchor=south] {\Small $#1$};
  \draw (1,-1) -- (1,0) node[anchor=south] {\Small $#2$};
  \draw (2,-1) -- (2,0) node[anchor=south] {\Small $#3$};
  \draw (3,-1) -- (3,0) node[anchor=south] {\Small $#4$};
  \draw (4,-1) -- (4,0) node[anchor=south] {\Small $#5$};
  \draw (0,-1) -- (2,-1);
  \draw (0,-1) .. controls +(0,-2) and +(0,-2) .. (4,-1);
  \draw (2,-1) .. controls +(0,-1) and +(0,-1) .. (3,-1);
\end{tikzpicture}%
}

\newcommand{\phigraph}[2]{%
\begin{tikzpicture}[scale=0.3, baseline={(0,-0.1)}, densely dashed]
  \draw (0,0) circle [radius=1];
  \draw (0,1) -- (0,2) node[anchor=south] {\Small $#1$};
  \draw (0,-1) -- (0,-2) node[anchor=north] {\Small $#2$};
\end{tikzpicture}%
}

\newcommand{\Igraph}[2]{%
\begin{tikzpicture}[scale=0.3, baseline={(0,0.4)}, densely dashed]
  \draw (0,0) -- (1,1);
  \draw (2,0) -- (1,1);
  \draw (1,1) -- (1,3);
  \draw (-0.5,3) -- (2.5,3) node[at start, anchor=south] {\Small $#1$} node[at end, anchor=south] {\Small $#2$};
\end{tikzpicture}%
}

\newcommand{\Pone}[2]{%
\begin{tikzpicture}[scale=0.25, baseline={(0,-0.2)}, densely dashed]
 \draw (0,0) circle [radius=4];
 \draw (0,-1)--(0,1);
 \draw (0,1)--(45:4);
 \draw (0,1)--(135:4);
 \draw (0,-1)--(-45:4);
 \draw (0,-1)--(-135:4);
 \draw (0,4)--(0,6) node[anchor=west] {\Small $#1$};
 \draw (1.5,2)--(2.5,0.5) node[anchor=north] {\Small $#2$};
\end{tikzpicture}
}
\newcommand{\Ptwo}[2]{%
\begin{tikzpicture}[scale=0.25, baseline={(0,-0.2)}, densely dashed]
 \draw (0,0) circle [radius=4];
 \draw (0,-1)--(0,1);
 \draw (0,1)--(45:4);
 \draw (0,1)--(135:4);
 \draw (0,-1)--(-45:4);
 \draw (0,-1)--(-135:4);
 \draw (0,4)--(0,6) node[anchor=west] {\Small $#1$};
 \draw (0,-4)--(0,-6) node[anchor=west] {\Small $#2$};
\end{tikzpicture}
}
\newcommand{\Pthree}[2]{%
\begin{tikzpicture}[scale=0.25, baseline={(0,-0.2)}, densely dashed]
 \draw (0,0) circle [radius=4];
 \draw (0,-1)--(0,1);
 \draw (0,1)--(45:4);
 \draw (0,1)--(135:4);
 \draw (0,-1)--(-45:4);
 \draw (0,-1)--(-135:4);
 \draw (0,4)--(0,6) node[anchor=west] {\Small $#1$};
 \draw (1.5,-2)--(2.5,-0.5) node[anchor=south] {\Small $#2$};
\end{tikzpicture}
}
\newcommand{\Pfour}[2]{%
\begin{tikzpicture}[scale=0.25, baseline={(0,-0.2)}, densely dashed]
 \draw (0,0) circle [radius=4];
 \draw (0,-1)--(0,1);
 \draw (0,1)--(45:4);
 \draw (0,1)--(135:4);
 \draw (0,-1)--(-45:4);
 \draw (0,-1)--(-135:4);
 \draw (0,4)--(0,6) node[anchor=west] {\Small $#1$};
 \draw (0,0)--(2,0) node[anchor=west] {\Small $#2$};
\end{tikzpicture}
}

\title[]{Torsion elements in the associated graded of the $Y$-filtration of the monoid of homology cylinders}

\author{Yuta Nozaki}
\address{
Faculty of Environment and Information Sciences, Yokohama National University \\
79-7 Tokiwadai, Hodogaya-ku, Yokohama, 240-8501 \\
Japan\vspace{-0.6em}}
\address{
WPI-SKCM$^2$, Hiroshima University \\
1-3-1 Kagamiyama, Higashi-Hiroshima, Hiroshima, 739-8526 \\
Japan}
\email{nozaki-yuta-vn@ynu.ac.jp}

\author{Masatoshi Sato}
\address{
Department of Mathematics and Data Science\\
Tokyo Denki University \\
5 Senjuasahi-cho, Adachi-ku, Tokyo 120-8551 \\
Japan}
\email{msato@mail.dendai.ac.jp}

\author{Masaaki Suzuki}
\address{Department of Frontier Media Science, Meiji University \\
4-21-1 Nakano, Nakano-ku, Tokyo, 164-8525 \\
Japan}
\email{mackysuzuki@meiji.ac.jp}

\makeatletter
\@namedef{subjclassname@2020}{\textup{2020} Mathematics Subject Classification}
\makeatother



\begin{document}
\maketitle

\begin{abstract}
Clasper surgery induces the $Y$-filtration $\{Y_n\mathcal{IC}\}_n$ over the monoid of homology cylinders, which serves as a $3$-dimensional analogue of the lower central series of the Torelli group of a surface.
In this paper, we investigate the torsion submodules of the associated graded modules of these filtrations.
To detect torsion elements, we introduce a homomorphism on $Y_n\mathcal{IC}/Y_{n+1}$ induced by the degree $n+2$ part of the LMO functor.
Additionally, we provide a formula that computes this homomorphism under clasper surgery, and use it to demonstrate that every non-trivial torsion element in $Y_6\mathcal{IC}/Y_7$ has order $3$.
\end{abstract}

\setcounter{tocdepth}{1}
\tableofcontents

\section{Introduction}
\label{sec:Introduction}
Let $\Sigma_{g,1}$ be a connected oriented compact surface of genus $g$ with one boundary component and let $\M=\M_{g,1}$ denote the mapping class group of $\Sigma_{g,1}$.
The mapping class group naturally acts on the first homology group $H_1(\Sigma_{g,1};\Z)$ and its kernel $\I=\I_{g,1}$ is called the Torelli group, which plays a central role in the study of $\M$ and the associated graded module $\bigoplus_{n=1}^\infty (\I(n)/\I(n+1))\otimes_{\Z}\Q$ is of particular interest.
Here, $\{\I(n)\}_n$ denotes the lower central series defined by $\I(n)=[\I(n-1),\I]$ and $\I(1)=\I$.

In \cite[Theorem~3]{Joh85III}, Johnson determined the abelianization $\I/\I(2)$ of $\I$ as an $\Sp(2g,\Z)$-module for $g\geq 3$.
Let $\tau_n\colon \I(n)/\I(n+1)\to H\otimes L_{n+1}$ denote the $n$th Johnson homomorphism, where $L_n$ denotes the degree $n$ part of the free Lie algebra generated by $H=H_1(\Sigma_{g,1};\Z)$.
In \cite[Theorem~10.1]{Hai97}, Hain determined $(\I(2)/\I(3))\otimes \Q$ for $g\geq 3$ and showed that the kernel of the induced homomorphism
\[
\tau_2\otimes\id_\Q\colon (\I(2)/\I(3))\otimes \Q\to (H\otimes L_3)\otimes \Q
\]
is of rank $1$, which is detected by the Casson invariant as explained in \cite{Mor89}.
He also gave a presentation of the associated graded Lie algebra $\bigoplus_{n=1}^\infty\I(n)/\I(n+1)\otimes\Q$ in \cite[Theorem~11.1]{Hai97}.
For $g\geq 6$, the $\Sp(2g,\Q)$-module $\I(3)/\I(4)\otimes\Q$ was determined by Morita~\cite[Proposition~6.3]{Mor99}.
Furthermore, Morita, Sakasai, and the third author~\cite[Theorem~1.2]{MSS20} proved that $\tau_n\otimes\id_\Q$ is an isomorphism when $n=4,5,6$ and $g$ is large enough.
Kupers and Randal-Williams~\cite[Theorem~B]{KuRW23} recently showed that the kernel of
\[
\tau_n\otimes\id_\Q\colon (\I(n)/\I(n+1))\otimes \Q\to (H\otimes L_{n+1})\otimes \Q
\]
is a trivial $\Sp(2g,\Q)$-module when $g\geq 3n$.
When $n\le 6$, it can also be proven by comparing the irreducible decompositions of the Torelli Lie algebra as an $\Sp(2g,\Q)$-representation in \cite[Section~7]{GaGe17} and of the images of the Johnson homomorphisms in \cite[Table~1]{MSS15AM}.

We next turn our attention to the torsion subgroup $\tor(\I(n)/\I(n+1))$.
As is well known, there are torsion elements of order $2$ in the abelianization $\I(1)/\I(2)$ detected by the Birman-Craggs homomorphisms.
On the other hand, $\I(2)/\I(3)$ was recently shown to be torsion-free in \cite{FMS24}.
Therefore, the existence of torsion elements in $\I(n)/\I(n+1)$ is a subtle problem.
In \cite{NSS22GT}, the authors proved that $\tor(\I(n)/\I(n+1))$ is non-trivial if $n=3,5$ and $g\geq n$.
Combining an argument in \cite{NSS22GT} with \cite[Theorem~B]{KuRW23} mentioned above, we prove the following stronger result in Section~\ref{subsec:Torelli}.

\begin{theorem}
\label{thm:tor_gr_Torelli}
When $n$ is odd and $g\ge3n$,
$\tor(\I(n)/\I(n+1))$ is non-trivial.
\end{theorem}

The key idea of \cite{NSS22GT} is to consider the monoid $\I\C=\I\C_{g,1}$ of homology cylinders over $\Sigma_{g,1}$.
A homology cylinder is a certain $3$-manifold with boundary and $\I\C$ can be regarded as a $3$-dimensional analogue of the Torelli group via a natural injective monoid homomorphism $\cc\colon \I\hookrightarrow \I\C$.
Goussarov~\cite{Gou99} and Habiro~\cite{Hab00C} independently introduced clasper surgery to study finite-type invariants of links and $3$-manifolds.
In particular, they introduced the $Y_n$-equivalence relation among homology cylinders and defined $Y_n\I\C$ as the submonoid of $\I\C$ consisting of homology cylinders being $Y_n$-equivalent to the trivial one.
Then we have the $Y$-filtration $\{Y_n\I\C\}_n$ on $\I\C$, which plays the role of the lower central series of $\I$.
More precisely, $\cc$ restricts to $\I(n) \to Y_n\I\C$ and induces a homomorphism $\cc_n\colon\I(n)/\I(n+1) \to Y_n\I\C/Y_{n+1}$ between abelian groups.

Goussarov and Habiro also observed that there is a surjective homomorphism $\ss_n\colon \A^c_n\to Y_n\I\C/Y_{n+1}$ induced by clasper surgery when $n\geq 2$.
Here, $\A^c_n$ is a $\Z$-module of connected Jacobi diagrams with $n$ trivalent vertices.
Since $\A^c_n$ is a purely combinatorial object, it suffices to determine the kernel of $\ss_n$ to reveal the group structure of $Y_n\I\C/Y_{n+1}$.
This strategy works well for small $n$.
In fact, $Y_n\I\C/Y_{n+1}$ is determined for $n=1,2$ by Massuyeau and Meilhan~\cite{MaMe03,MaMe13} and for $n=3,4$ by the authors~\cite{NSS22GT,NSS22JT}.
As a corollary, the Goussarov-Habiro conjecture is true for the $Y_{n+1}$-equivalence when $n\leq 4$, and therefore $Y_n\I\C/Y_{n+1}$ attracts considerable attention.
We refer the reader to \cite[Section~3.5]{Mas21} and \cite{HaMa12} for a survey.
In this paper, we partially investigate $Y_n\I\C/Y_{n+1}$ for $n=5,6,7$ in Section~\ref{sec:Yn/Yn+1}.

Cheptea, Habiro, and Massuyeau~\cite{CHM08} constructed the LMO functor as an extension of the Le-Murakami-Ohtsuki invariant~\cite{LMO98} of closed $3$-manifolds to certain $3$-dimensional cobordisms.
As an application, they proved that the surgery map $\ss_n$ is an isomorphism over $\Q$ for $n\geq 1$, while $\ss_n$ itself is not necessarily injective.
This implies that the kernel $\Ker \ss_n$ is contained in the torsion subgroup $\tor \A^c_n$, and thus it seems to be difficult to detect non-trivial elements of $\Ker \ss_n$, let alone determine $\Ker \ss_n$ for large $n$.
Conant, Schneiderman, and Teichner~\cite{CST16} studied the homology cobordism group of homology cylinders, and as a consequence, they revealed that $Y_n\I\C/Y_{n+1}$ has torsion elements of order $2$ when $n$ is odd.
The authors also found torsion elements of order $2$ in \cite{NSS22GT,NSS22JT}.
The key ingredient of \cite{NSS22GT,NSS22JT} is a homomorphism $\bar{z}_{n+1}\colon Y_n\I\C/Y_{n+1}\to \A^c_{n+1}\otimes\Q/\Z$ induced by the degree $n+1$ term of the LMO functor.
A formula of $\bar{z}_{n+1}$ for clasper surgery is also given in \cite{NSS22GT}, which enables us to detect torsion elements of order $2$.

In this paper, we introduce a homomorphism
\[
\zz_{n+2}\colon Y_n\I\C/Y_{n+1}\to \A^c_{n+2}\otimes\Q/\tfrac{1}{2}\Z
\]
induced by the degree $n+2$ term of the LMO functor and give a formula of $\zz_{n+2}$ for clasper surgery in Theorem~\ref{thm:formula}.
As an application, we can find torsion elements with completely different properties from those previously found.
Recall here that the non-triviality of $\tor(Y_n\I\C/Y_{n+1})$ is known only for odd integers $n\geq1$ and that the orders of torsion elements are even. 
Then, it is natural to ask about the existence of torsion elements of odd order and the existence of torsions in $Y_n\I\C/Y_{n+1}$ with $n$ even.
The next consequence of Theorem~\ref{thm:formula} answers both of the questions affirmatively.

\begin{theorem}
\label{thm:Y6/Y7}
The abelian group $\tor(Y_6\I\C/Y_7)$ is isomorphic to $(\Z/3\Z)^r$, where $g\geq 0$ and $\binom{2g}{2}\leq r\leq 4g^2$.
\end{theorem}

We also investigate the structure of the kernel $\Ker\ss_n$ of the surgery map $\ss_n$.
To study $\Ker\ss_n$, it is convenient to use the decomposition $\A^c_n = \bigoplus_{l\geq 0}\A^c_{n,l}$ with respect to the first Betti number $l$ of Jacobi diagrams.
For instance, in \cite{NSS22GT,NSS22JT}, it works very well for small $n$.
Indeed, the inclusion $\bigoplus_{l\geq 0}\Ker\ss_{n,l} \subset \Ker\ss_{n}$ is an equality if $n\leq 4$.
On the other hand, we show that the above decomposition is not enough to study $\Ker\ss_n$.

\begin{theorem}
\label{thm:deg-by-deg}
When $g\geq 1$, the inclusion $\bigoplus_{l\geq 0}\Ker\ss_{7,l} \subset \Ker\ss_{7}$ is strict.
In fact,
for distinct $a,b\in \{1^\pm, \dots, g^\pm\}$,
\[
O(a,a,a,b,a,a,a)+O(b,a,a,a,a,a,b)+\theta(a;a;a,b,a)+\theta(a,a,a;a;b)
\]
lies in the gap, where $O(a_1,a_2,a_3,\ldots, a_n)$ and $\theta(a_1,\dots,a_p;b_1,\dots,b_q;c_1,\dots,c_r)$ are respectively Jacobi diagrams
\[
\begin{tikzpicture}[scale=0.25, baseline={(0,0.1)}, densely dashed]
 \draw (0,0) circle [radius=2];
 \draw (90:2)--(90:4) node[anchor=south] {$a_1$};
 \draw (50:2)--(50:4) node[anchor=south west] {$a_2$};
 \draw (10:2)--(10:4) node[anchor=west] {$a_3$};
 \draw[loosely dotted, very thick] (-20:4) arc (-20:-50:4);
 \draw (130:2)--(130:4) node[anchor=south east] {$a_n$};
 \draw[loosely dotted, very thick] (160:4) arc (160:190:4);
\end{tikzpicture}
\qquad and \qquad
\begin{tikzpicture}[scale=0.3, baseline={(0,-0.2)}, densely dashed]
 \draw (0,0) circle [radius=3];
 \draw (-2.8,-1) -- (2.8,-1);
 \draw (135:3) -- (135:4) node[anchor=south] {$a_1$};
 \node at (0,4) {$\cdots$};
 \draw (45:3) -- (45:4) node[anchor=south] {$a_p$};
 \draw (-1.5,-1) -- (-1.5,0) node[anchor=south] {$b_1$};
 \node at (0,0) {$\cdots$};
 \draw (1.5,-1) -- (1.5,0) node[anchor=south] {$b_q$};
 \draw (-135:3) -- (-135:4) node[anchor=north] {$c_1$};
 \node at (0,-4) {$\cdots$};
 \draw (-45:3) -- (-45:4) node[anchor=north] {$c_r$};
\end{tikzpicture}
\]
for $a_i, b_j, c_k \in \{1^\pm,\ldots,g^\pm\}$.
\end{theorem}

Theorem~\ref{thm:deg-by-deg} means that there exists a non-trivial relation between claspers with the same degree but with different first Betti numbers, which seems to be new and interesting.
Note that the STU relation (cf.~\cite[Figure~45]{Hab00C}) is a relation between claspers with different degrees.

\subsection*{Organization of this paper}
In Section~\ref{sec:Preliminaries}, we will review the basic definitions concerning the LMO functor and prove Theorem~\ref{thm:tor_gr_Torelli}.
Section~\ref{sec:homomorphisms} is devoted to the proof of Theorem~\ref{thm:formula} which is our main result.
As an application, we obtain Theorem~\ref{thm:Y6/Y7}.
In Section~\ref{sec:Yn/Yn+1}, we will observe $Y_7\I\C/Y_8$ and show Theorem~\ref{thm:deg-by-deg}.

\subsection*{Acknowledgments}
The authors would like to thank Nariya Kawazumi for his question that encouraged us to pursue the degree $n+2$ term of the LMO functor.
They also thank Katsumi Ishikawa for informing us of the existence of $3$-torsions in a module of Jacobi diagrams (cf.~Remark~\ref{rem:Ishikawa}).
This study was supported in part by JSPS KAKENHI Grant Numbers JP20K14317, JP23K12974, JP22K03298, and 20K03596.

\section{Preliminaries}
\label{sec:Preliminaries}
In this section, we review the basic definitions concerning the LMO functor.
We refer the reader to \cite{CHM08} and \cite[Section~2]{NSS22GT} for more details about the LMO functor.
In Section~\ref{subsec:Torelli}, the proof of Theorem~\ref{thm:tor_gr_Torelli} will be given.

\subsection{Homology cylinders}
Let $M$ be a connected oriented compact $3$-manifold with boundary and let $m\colon \partial(\Sigma_{g,1}\times[-1,1])\to \partial M$ be an orientation-preserving homeomorphism.
We write $m_{+}$ and $m_{-}$ for the restrictions of $m$ to $\Sigma_{g,1}\times\{1\}$ and $\Sigma_{g,1}\times\{-1\}$, respectively.
A pair $(M,m)$ is called a \emph{homology cylinder} over $\Sigma_{g,1}$ if the induced maps $(m_{\pm})_\ast\colon H_\ast(\Sigma_{g,1};\Z)\to H_\ast(M;\Z)$ are the same isomorphism.
Two pairs $(M,m)$ and $(M',m')$ are equivalent if there exists an orientation-preserving homeomorphism $\phi\colon M\to M'$ such that $\phi\circ m=m'$.
Let $\I\C=\I\C_{g,1}$ denote the monoid of equivalent classes of homology cylinders over $\Sigma_{g,1}$.
Here the product of $(M,m)$ and $(M',m')$ is defined by stacking $(M',m')$ on $(M,m)$, that is, $(M\cup_{m_+=m'_-} M', m_-\cup m'_+)$.

A homology cylinder is a special case of a Lagrangian cobordism which is a 3-manifold whose boundary consists of $\Sigma_{g_{+},1}$, $\Sigma_{g_{-},1}$ and annulus satisfying some homological condition (see \cite[Definition~2.2]{CHM08} for the precise definition).

\subsection{Bottom-top tangles}
\label{subsec:bottom-top_tangle}
For a positive integers $g$, fix $g$ pairs of points $(p_1,q_1),\dots,(p_g,q_g)$ in $[-1,1]^2$ uniformly along the first coordinate.
We call a homology cylinder over $[-1,1]^2$ a \emph{homology cube}.
Let $B=(B,m)$ be a homology cube and identify $\partial B$ with $\partial[-1,1]^3$ via $m$.
For non-negative integers $g_+$ and $g_-$, let $\gamma=(\gamma^+,\gamma^-)$ be a framed oriented tangle in $B$ with $g_+$ top components $\gamma^{+}_1,\dots, \gamma^{+}_{g_+}$ and $g_-$ bottom components $\gamma^{-}_1,\dots, \gamma^{-}_{g_-}$ such that each $\gamma_j^{-}$ runs from $q_j\times\{-1\}$ to $p_j\times\{-1\}$ and each $\gamma_j^{+}$ runs from $p_j\times\{1\}$ to $q_j\times\{1\}$.
A pair $(B,\gamma)$ is called a \emph{bottom-top tangle} of type $(g,h)$ in $B$.
In Figure~\ref{fig:bt_tangles}, we give examples of bottom-top tangles in $[-1,1]^3$.
Note here that we use the blackboard framing convention throughout this paper.

\begin{figure}[h]
 \centering
 \includegraphics[width=\textwidth]{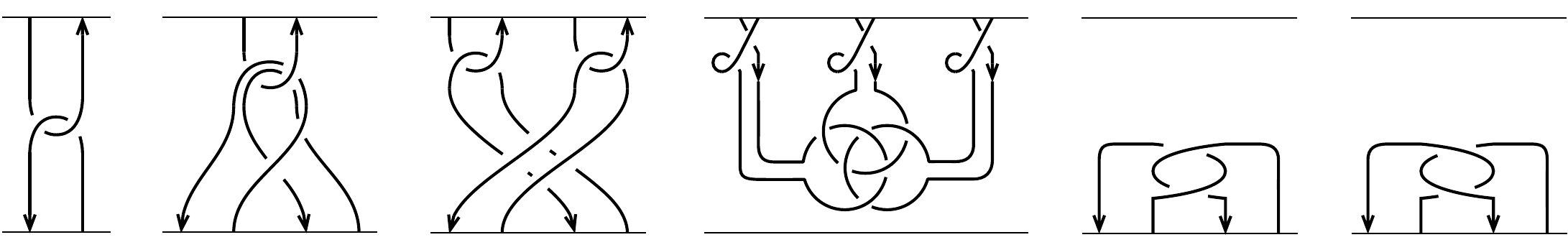}
 \caption{Bottom-top tangles $\Id_1$, $\mu$, $\psi$, $Y$, $c$, and $c'$.}
 \label{fig:bt_tangles}
\end{figure}

Let $(B,\gamma)$ be a bottom-top tangle of type $(g_+,g_-)$ in a homology cube $B$.
Then we obtain a cobordism $(M,m)$ from $\Sigma_{g_+,1}$ to $\Sigma_{g_-,1}$ by digging $B$ along the tangle $\gamma$.
Here the homeomorphism $m\colon \Sigma_{g_+,1} \cup (S^1\times[-1,1]) \cup \Sigma_{g_-,1} \to \partial M$ is uniquely determined (up to isotopy) by the framing of $\gamma$.
See \cite[Theorem~2.10]{CHM08} for details.
Assume that $g_+=g_-=g$ and that the linking matrix $\Lk_B(\gamma)$ of $\gamma$ in $B$ is
\[
\begin{pmatrix}
O_g&I_g\\
I_g&O_g
\end{pmatrix},
\]
where $O_g$ and $I_g$ are the zero matrix and identity matrix of size $g$, respectively.
In this case, we obtain a homology cylinder over $\Sigma_{g,1}$ as mentioned in \cite[Section~8.1]{CHM08}.

\begin{figure}[h]
 \centering
 \includegraphics[width=0.6\textwidth]{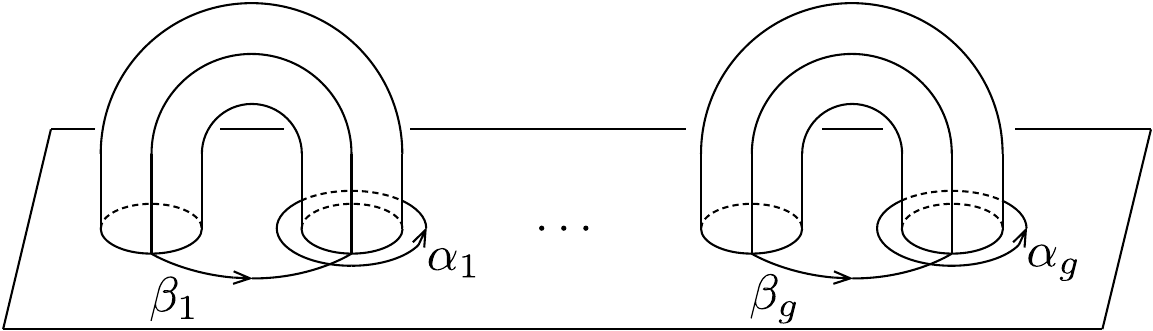}
 \caption{Oriented simple closed curves $\alpha_i$ and $\beta_i$ on $\Sigma_{g,1}$.}
 \label{fig:Sigma_g1}
\end{figure}

Conversely, let $M$ be a cobordism from $\Sigma_{g_+,1}$ to $\Sigma_{g_-,1}$ satisfying some homological condition.
Then we obtain a homology cube $B$ by attaching $3$-dimensional $2$-handles to the boundary of $M$ along each of $\beta_1,\ldots,\beta_{g_+}$ in the top surface and $\alpha_1,\dots,\alpha_{g_-}$ in the bottom surface.
Here, $\alpha_1, \beta_1,\ldots, \alpha_g,\beta_g$ are the oriented simple closed curves in Figure~\ref{fig:Sigma_g1}.
Moreover, letting $\gamma$ be the co-cores of these 2-handles, we obtain a bottom-top tangle $(B,\gamma)$.

Under this correspondence, the composition of cobordisms $M$ from $\Sigma_{g,1}$ to $\Sigma_{f,1}$ and $M'$ from $\Sigma_{h,1}$ to $\Sigma_{g,1}$ induces a composition of bottom-top tangles $\gamma$ of type $(g,f)$ and $\gamma'$ of type $(h,g)$ as described in \cite[Section~2.3]{CHM08}.
We denote the composition by $\gamma\circ\gamma'$, which is of type $(h,f)$, and note that the composition is not just concatenation.

\subsection{Jacobi diagrams}
\label{subsec:Jacobi}
Let $X$ be a (possibly disconnected) oriented compact $1$-manifold and let $C$ be a finite set (of colors or labels).
A \emph{Jacobi diagram} based on $(X,C)$ is a uni-trivalent graph such that each univalent vertex is attached to $X$ or colored by an element of $C$, and for each trivalent vertex $v$ a cyclic order of the half-edges incident to $v$ is equipped.
We use dashed lines for uni-trivalent graphs and solid lines for $X$ as in \cite{CHM08}.
Let $\A(X,C)$ denote the $\Z$-module generated by Jacobi diagrams subject to the AS, IHX, STU, and self-loop relations:
\begin{gather*}
\Ygraph{0.5}{}{}{}
+\!\!
\begin{tikzpicture}[scale=0.5, baseline={(0,-0.1)}, densely dashed]
  \coordinate (origin) at (0,0);
  \draw (origin) .. controls +(1,0.5) and +(1,-0.5) .. (-1,1) node[at end, anchor=south east] {};
  \draw (origin) .. controls +(-1,0.5) and +(-1,-0.5) .. (1,1) node[at end, anchor=south west] {};
  \draw (origin) -- (0,-1) node[at end, anchor=north] {};
\end{tikzpicture}%
\!\!=0,
\qquad
\begin{tikzpicture}[scale=0.5, baseline={(0,-0.1)}, densely dashed]
  \draw (-1,1) -- (1,1) ;
  \draw (-1,-1) -- (1,-1) ;
  \draw (0,1) -- (0,-1) ;
\end{tikzpicture}%
-\ 
\begin{tikzpicture}[scale=0.5, baseline={(0,-0.1)}, densely dashed]
  \draw (-1,1) -- (-1,-1) ;
  \draw (1,1) -- (1,-1) ;
  \draw (-1,0) -- (1,0) ;
\end{tikzpicture}%
\ +
\begin{tikzpicture}[scale=0.5, baseline={(0,-0.1)}, densely dashed]
  \draw (-1,1) -- (1,-1);
  \draw (-1,-1) -- (1,1);
  \draw (-0.5,-0.5) -- (0.5,-0.5);
\end{tikzpicture}%
=0,
\\
\begin{tikzpicture}[scale=0.5, baseline={(0,0.3)}, densely dashed]
\draw[->,solid] (0,0) -- (4,0);
\draw (2,0) -- (2,1) ;
\draw (2,1) -- (1,2) ;
\draw (2,1) -- (3,2) ;
\end{tikzpicture}%
\ -\ 
\begin{tikzpicture}[scale=0.5, baseline={(0,0.3)}, densely dashed]
\draw[->,solid] (0,0) -- (4,0);
\draw (1,0) -- (1,2) ;
\draw (3,0) -- (3,2) ;
\end{tikzpicture}%
\ +\ 
\begin{tikzpicture}[scale=0.5, baseline={(0,0.3)}, densely dashed]
\draw[->,solid] (0,0) -- (4,0);
\draw (1,0) -- (3,2) ;
\draw (3,0) -- (1,2) ;
\end{tikzpicture}%
=0,
\qquad
\begin{tikzpicture}[scale=0.5, baseline={(0,-0.1)}, densely dashed]
  \draw (0,0) circle [radius=1];
  \draw (1,0) -- (2,0) node {};
\end{tikzpicture}%
=0,
\end{gather*}
where the rest of the diagrams are the same in each relation.
For a Jacobi diagram $J$, we define the \emph{degree} $\deg J$ to be half the number of vertices and the \emph{internal degree} $\ideg J$ by the number of trivalent vertices.
Note that the degree is preserved by the relations in general and the internal degree is preserved if $X$ is empty.
When $X=\emptyset$, we simply write $\A(C)$ for $\A(\emptyset,C)$ and we have $\A(C) = \bigoplus_{i\geq 0} \A_i(C)$, where $\A_i(C)$ denotes the submodule generated by Jacobi diagrams of $\ideg=i$.
Let $\widehat{\A}(C)_\Q$ denote the completion of $\A(C)_\Q=\A(C)\otimes \Q$ with respect to $\ideg$, that is, $\widehat{\A}(C)_\Q = \prod_{i\geq 0} \A_i(C)_\Q$.
It is known that $\widehat{\A}(C)_\Q$ has a structure of a complete Hopf algebra (see \cite[Section~3.1]{CHM08}), and the primitive elements coincides with the submodule $\widehat{\A}^c(C)_\Q$ generated by connected Jacobi diagrams.
Then, the maps $\exp=\exp_\sqcup$ and $\log=\log_\sqcup$ with respect to the disjoint union $\sqcup$ are defined in the usual manner.

A connected Jacobi diagram without trivalent vertices is called a \emph{strut} and let $\A^Y(C)$ denote the quotient of $\A(C)$ by declaring any diagram containing a strut to be zero.
The image of $x \in \A(C)$ under the projection $\A(C)\twoheadrightarrow \A^Y(C)$ is denoted by $x^Y$.
Let $J\in \A(\{1^+,\dots,q^+,1^-,\dots,p^-\}))$ and $J'\in \A(\{1^+,\dots,r^+,1^-,\dots,q^-\}))$ be Jacobi diagrams.
The composition $J\circ J'\in \A(\{1^+,\dots,r^+,1^-,\dots,p^-\}))$ is defined to be the sum of all ways of gluing the $i^+$-colored vertices of $J$ to the $i^-$-colored vertices of $J'$ for all $i\in \{1,\dots,q\}$.
We refer the reader to \cite[Section~4.2]{CHM08} or \cite[Section~2.6]{NSS22GT} for details.
Moreover, the linear extension of this composition is defined among \emph{top-substantial} Jacobi diagrams, that is, Jacobi diagrams without struts both of whose vertices are colored by $\{1^+,2^+,\dots\}$.

In this paper, we mainly consider the case $(X,C)=(\emptyset, \{1^\pm,\dots,g^\pm\})$, so we simply write $\A$ for $\A(\emptyset, \{1^\pm,\dots,g^\pm\})$.

\subsection{The LMO functor}
Cheptea, Habiro, and Massuyeau introduced the LMO functor as a functorial extension of the LMO invariant.
The LMO functor $\Ztilde\colon \LCob_q\to \tsA$ is a functor from a certain category of cobordisms to a certain category of Jacobi diagrams, which can be used as an invariant of cobordisms.
Let us first recall these two categories following \cite[Section~4]{CHM08}.
We write $\Mag(\bullet)$ for the free magma generated by a letter $\bullet$, for example, $(\bullet(\bullet\bullet))(\bullet\bullet) \in \Mag(\bullet)$.
A Lagrangian $q$-cobordism is a Lagrangian cobordism from $\Sigma_{g_+,1}$ to $\Sigma_{g_-,1}$ together with $w_+, w_-\in \Mag(\bullet)$ with $|w_\pm|=g_\pm$, where $|w|$ denotes the length $w\in \Mag(\bullet)$.
Let $\LCob_q$ denote the category whose objects are elements of $\Mag(\bullet)$ and whose morphisms from $w_+$ to $w_-$ are Lagrangian $q$-cobordisms from $\Sigma_{|w_+|,1}$ to $\Sigma_{|w_-|,1}$.
In this paper, we regard homology cylinders as Lagrangian $q$-cobordisms with $w_{+}=w_{-}=(\cdots((\bullet\bullet)\bullet) \cdots \bullet) \in \Mag(\bullet)$.
Let $\tsA$ denote the category whose objects are non-negative integers and whose morphisms from $n_+$ to $n_-$ are infinite sums of top-substantial Jacobi diagrams, where the composition is given by gluing univalent vertices colored by $i^+$ and $i^-$ for each $i$.
See \cite[Section~4.2]{CHM08} for the precise definition.

Next, we briefly recall the definition of the LMO functor.
For an object $w\in \Mag(\bullet)$, we define $\Ztilde(w)=|w|$.
Let $(M,m)$ be a Lagrangian $q$-cobordism from $w_+$ to $w_-$.
As in Section~\ref{subsec:bottom-top_tangle}, we obtain a bottom-top tangle $(B,\gamma)$ together with $w_+$ and $w_-$, which is called a \emph{bottom-top $q$-tangle}.
Since $B$ is a homology cube, it is homeomorphic to the $3$-manifold $[-1,1]^3_L$ obtained by Dehn surgery along some framed link $L$ in $[-1,1]^3$, and the tangle in $[-1,1]^3$ corresponding to $\gamma\subset B$ is again denoted by $\gamma$.
Now, by choosing an associator, the Kontsevich invariant of the framed tangle $\gamma\cup L$ in $[-1,1]^3$ is defined.
Throughout this paper, we mainly use an (even) rational associator following \cite{CHM08}.
Applying the Aarhus integral to the resulting value and normalizing it suitably, we obtain a series of top-substantial Jacobi diagrams that is independent of the choice of $L$.
This procedure defines $\Ztilde$ at the level of morphisms.
In particular, $\Ztilde$ induces the LMO homomorphism $\I\C\to \widehat{\A}_\Q$ on the monoid $\I\C$ of homology cylinders over $\Sigma_{g,1}$.

Massuyeau \cite{Mas12} proved that the tree part of the LMO functor corresponds to the total Johnson homomorphism.
The authors \cite{NSS23} showed that the $1$-loop part is related to a non-commutative Reidemeister-Turaev torsion.

\subsection{Claspers}
A graph clasper in $M$ is an embedded surface consisting of annuli, disks, and bands such that each disk is connected with three bands and each annulus is connected with one band.
We can obtain a framed link from a graph clasper $G$ according to \cite{Hab00C} and perform Dehn surgery along it.
This procedure is called \emph{clasper surgery} and the resulting $3$-manifold is denoted by $M_G$.
For a graph clasper $G$, its \emph{degree} $\deg G$ is defined to be the number of disks of $G$.
Two homology cylinders $M$ and $M'$ are said to be \emph{$Y_n$-equivalent} if there exist disjoint graph claspers $G_1,\dots,G_k$ of degree $n$ in $M$ satisfying $M_{G_1\sqcup\cdots\sqcup G_k} = M'$.
Let $Y_n\I\C$ denote the submonoid consisting of homology cylinders over $\Sigma_{g,1}$ being $Y_n$-equivalent to the trivial one $\Sigma_{g,1}\times[-1,1]$.
Then we have a descending series $\I\C = Y_1\I\C \supset Y_2\I\C \supset \cdots$ of submonoids, which is called the \emph{$Y$-filtration} on $\I\C$.
The quotient $Y_n\I\C/Y_{n+1}$ of $Y_n\I\C$ by the $Y_{n+1}$-equivalence is known to be a finitely generated abelian group (see \cite[Section~8.5]{Hab00C}).

For a Jacobi diagram $J$ in $\A^c_n$, we obtain a graph clasper $G(J)$ of degree $n$ in $\Sigma_{g,1}\times[-1,1]$ as follows.
First, replace univalent vertices, trivalent vertices, and edges of $J$ with annuli, disks, and bands, respectively.
Next, embed the resulting surface according to labels of univalent vertices of $J$.
See \cite{Hab00C} or \cite{NSS22GT} for details.
It is shown that $(\Sigma_{g,1}\times[-1,1])_{G(J)}$ is well-defined up to $Y_{n+1}$-equivalence, and thus we have a homomorphism $\ss_n\colon \A^c_n\to Y_n\I\C/Y_{n+1}$.

For a (possibly disconnected) graph clasper $G$ in $M$, define $[M,G] \in \Z\I\C$ by
$
[M,G]=\sum_{G'\subset G} (-1)^{|G'|}M_{G'},
$
where $G'$ runs over unions of connected components of $G$ and $|G'|$ denotes the number of connected components of $G'$.
Let $\F_n\I\C$ denote the submodule of $\Z\I\C$ generated by elements $[M,G]$ for $M\in \I\C$ and graph claspers $G$ of degree $n$.
This gives a descending series $\Z\I\C \supset \F_1\I\C \supset \F_2\I\C \supset \cdots$.
We then have the homomorphism $\S_n\colon \A^Y_n\to \F_n\I\C/\F_{n+1}\I\C$ defined by $\S_n(J)=[\Sigma_{g,1}\times[-1,1], G(J)]$.

The homomorphisms $\ss_n$ and $\S_n$ are known to be surjective if $n\geq 2$.
Furthermore, $\ss_n\otimes\id_\Q$ and $\S_n\otimes\id_\Q$ are isomorphisms for $n\geq 1$.
In fact, the degree $n$ part of the LMO functor induces homomorphisms $Y_n\I\C/Y_{n+1}\to \A^c_n\otimes\Q$ and $\F_n\I\C/\F_{n+1}\I\C\to \A^Y_n\otimes\Q$, which give the inverses up to sign \cite[Theorem~7.11]{CHM08}.

\subsection{Torsion elements of $\I(n)/\I(n+1)$}
\label{subsec:Torelli}
In \cite{NSS22GT}, the authors constructed a homomorphism $\bar{z}_{n+1}\colon Y_n\I\C/Y_{n+1}\to \A^c_{n+1}\otimes \Q/\Z$ induced by $\Ztilde_{n+1}$ and gave a formula for clasper surgery in terms of Jacobi diagrams.
As an application of $\bar{z}_{n+1}$ and \cite{KuRW23}, we here prove Theorem~\ref{thm:tor_gr_Torelli}.

\begin{proof}[Proof of Theorem~\ref{thm:tor_gr_Torelli}]
The authors showed in \cite[Theorem~1.2]{NSS22GT} that the composition of
\[
\bar{z}_{2n}=(\log \Ztilde^Y)_{2n}\colon Y_{2n-1}\I\C/Y_{2n}\to \A_{2n}^c\otimes \Q/\Z
\]
and the natural homomorphism 
\[
\mathfrak{c}_{2n-1}\colon\I(2n-1)/\I(2n) \to Y_{2n-1}\I\C/Y_{2n}
\]
is non-trivial.
It is also non-trivial when restricted to the kernel $\Ker \tau_{2n-1}\subset \I(2n-1)/\I(2n)$ of the $(2n-1)$st Johnson homomorphism $\tau_{2n-1}$.
For example, let $x=O(1^+,2^+, \ldots, n^+, \ldots, 2^+,1^+)\in\A_{2n-1}^c$. 
By \cite[Lemma~6.2]{NSS22GT}, there exists $\varphi\in\I(2n-1)$ such that $\mathfrak{c}_{2n-1}(\varphi)=\ss_{2n-1}(x)\in Y_{2n-1}\I\C/Y_{2n}$.
Moreover, as in the paragraph just after \cite[Proof of Theorem~1.2]{NSS22GT}, we have $\varphi\in\Ker\tau_{2n-1}$.
Let $\psi$ be the mapping class which sends $\beta_i$ to $\beta_{i+1}$ for $1\le i\le n$,
where $\{\alpha_i,\beta_i\}_{i=1}^g$ denotes the basis of $\pi_1\Sigma_{g,1}$ in Figure~\ref{fig:Sigma_g1}
and $\beta_{g+1}=\beta_1$.
Setting $y=O(2^+,3^+, \ldots, (n+1)^+, \ldots, 3^+,2^+)$, we have
\[
\mathfrak{c}_{2n-1}(\psi)\circ\ss_{2n-1}(x)\circ\mathfrak{c}_{2n-1}(\psi^{-1})=\ss_{2n-1}(y)\in Y_{2n-1}\I\C/Y_{2n}.
\]
In \cite[Theorem~1.1]{NSS22GT}, we describe the composition
\[
\bar{z}_{2n}\circ\ss_{2n-1}\colon  \A_{2n-1}^c\to \A_{2n}^c\otimes \Q/\Z
\]
explicitly in terms of an operation on Jacobi diagrams.
In particular, we have $\bar{z}_{2n}(\ss_{2n-1}(y))\ne\bar{z}_{2n}(\ss_{2n-1}(x))$.
Thus, we obtain
\[
\bar{z}_{2n}(\mathfrak{c}_{2n-1}(\psi\circ\varphi\circ\psi^{-1}))\ne \bar{z}_{2n}(\mathfrak{c}_{2n-1}(\varphi)).
\]
As explained in Section~\ref{sec:Introduction},
$\Ker(\tau_n\otimes\id_\Q)\subset \I(n)/\I(n+1)\otimes\Q$
is a trivial $\Sp(2g,\Q)$-module when $3n\le g$ as shown in  \cite[Theorem~B]{KuRW23}.
Since $\varphi\in \Ker\tau_{2n-1}$ and the $\Sp(2g,\Q)$-action on $\I(n)/\I(n+1)$ is induced by the conjugacy action of $\M$ on $\I(n)$, we have
\[
\psi\circ\varphi\circ\psi^{-1}=\varphi\in\I(2n-1)/\I(2n)\otimes \Q.
\]
Thus, 
the commutator $[\psi,\varphi]\in\I(2n-1)/\I(2n)$
is a non-trivial torsion element.
\end{proof}

\begin{remark}
In \cite{FaMa24}, Faes and Massuyeau constructed a homomorphism $\mathcal{R}$ from $\K$ to some torsion module which factors through $\K/\I(4)$, and constructed an element $\varphi'\in\I(3)/\I(4)$ such that $\mathcal{R}(\varphi')\ne0$.
Using \cite[Theorem~1.2]{MSS20}, it is shown to be a torsion element by an argument similar to the proof of Theorem~\ref{thm:tor_gr_Torelli}.
\end{remark}

\section{Homomorphisms induced by the LMO functor}
\label{sec:homomorphisms}
In this section, we introduce two homomorphisms $\ZZ_{n+2}$ and $\zz_{n+2}$ via the LMO functor and investigate their properties, which play a crucial role in this paper.

\subsection{Definitions of $\ZZ_{n+2}$ and $\zz_{n+2}$}

\begin{definition}
For a positive integer $n$, define a homomorphism 
\[
\ZZ_{n+2}\colon \F_n\I\C/\F_{n+1}\I\C \to \A^Y_{n+2}\otimes_\Z \Q \twoheadrightarrow \A^Y_{n+2}\otimes_\Z \Q/\tfrac{1}{2}\Z
\]
by $\ZZ_{n+2}([x])=\Ztilde^Y_{n+2}(x)$.
Also, define a homomorphism
\[
\zz_{n+2}\colon Y_n\I\C/Y_{n+1} \to \A^c_{n+2}\otimes_\Z \Q \twoheadrightarrow \A^c_{n+2}\otimes_\Z \Q/\tfrac{1}{2}\Z
\]
by $\zz_{n+2}([M])=(\log\Ztilde^Y(M))_{n+2}$, where $\log=\log_\sqcup$ as in Section~\ref{subsec:Jacobi}.
\end{definition}

The previous result \cite[Theorem~1.1]{NSS22GT} and the surjectivity of the map $\S_{n+1}$ induced by clasper surgery imply $\Ztilde^Y_{n+2}(\F_{n+1}\I\C) \subset \Im\iota_{n+2}$ for $n\geq 1$, where $\iota_{n}$ is the induced homomorphism appearing in the exact sequence
\[
\A^Y_{n}\otimes\tfrac{1}{2}\Z \xrightarrow{\iota_{n}} \A^Y_{n}\otimes\Q \to \A^Y_{n}\otimes\Q/\tfrac{1}{2}\Z \to 0.
\]
Hence, the map $\ZZ_{n+2}$ is well-defined.
To see the well-definedness of $\zz_{n+2}$, it suffices to show
\[
(\log\Ztilde^Y(M))_{n+2}\equiv (\log\Ztilde^Y(M_G))_{n+2} \mod \tfrac{1}{2}\Z
\]
for $M\in Y_n\I\C$ and a connected graph clasper $G$ of degree $n+1$.
Let $x_d = (\log\Ztilde^Y(M))_{d}$ and $y_d = (\log\Ztilde^Y(M_G))_{d}$.
Since $M-M_G=[M,G]\in \F_{n+1}\I\C$ and
\[
\Ztilde^Y_{d}(\F_{n+1}\I\C)
\begin{cases}
=\{0\} & \text{if $1\leq d\leq n$,}\\
\subset \Im(\A^Y_{n} \to \A^Y_{n}\otimes\Q) & \text{if $d=n+1$,}\\
\subset \Im\iota_{n+2} & \text{if $d=n+2$,}
\end{cases}
\]
we have 
\[
x_d
\begin{cases}
= y_d & \text{if $1\leq d\leq n$,}\\
\equiv y_d \mod \Z & \text{if $d=n+1$.}
\end{cases}
\]
It follows that
\begin{align*}
\Ztilde^Y_{n+2}([M,G])
&=\bigl(\Ztilde^Y(M)-\Ztilde^Y(M_G)\bigr)_{n+2}\\
&=\bigl(\exp(x_1+\cdots+x_{n+1}+x_{n+2}+\cdots)-\exp(y_1+\cdots+y_{n+1}+y_{n+2}+\cdots)\bigr)_{n+2}\\
&\equiv x_{n+2}-y_{n+2}\bmod\Z.
\end{align*}
Thus, we obtain the desired equality modulo $\frac{1}{2}\Z$.

\begin{remark}
\label{rem:logZ}
For $M\in Y_n\I\C$, noting that $\Ztilde_k^Y(M)=0$ for $1\le k<n$, we have
\[
(\log\Ztilde^Y(M))_{n+2}=
\begin{cases}
 \Ztilde^Y_{n+2}(M) & \text{if $n\geq 3$,} \\
 \Ztilde^Y_{4}(M)-\frac{1}{2}\Ztilde^Y_{2}(M)\sqcup\Ztilde^Y_{2}(M) & \text{if $n=2$,} \\
 \Ztilde^Y_{3}(M)-\Ztilde^Y_{1}(M)\sqcup\Ztilde^Y_{2}(M) +\frac{1}{3}\Ztilde^Y_{1}(M)^{\sqcup 3} & \text{if $n=1$.}
\end{cases}
\]
Since the coefficients of $\Ztilde^Y_{2}(M)$ lie in $\tfrac{1}{2}\Z$ if $n=1$ and in $\Z$ if $n=2$, one obtains the following equality in $\A^Y_{n+2}\otimes\Q/\tfrac{1}{2}\Z$:
\[
\zz_{n+2}([M])=
\begin{cases}
 \Ztilde^Y_{n+2}(M) & \text{if $n\geq 2$,} \\
 \Ztilde^Y_{3}(M) +\frac{1}{3}\Ztilde^Y_{1}(M)^{\sqcup 3} & \text{if $n=1$.}
\end{cases}
\]
\end{remark}

\begin{remark}
In \cite{NSS22GT}, we construct a homomorphism
\[
\bar{z}_{n+1}\colon Y_n\I\C/Y_{n+1}\to \A_{n+1}^c\otimes \Q/\Z
\]
which does not depend on the choice of an even rational associator $\Phi$. 
The homomorphism $\zz_{n+2}$ is also independent of such a $\Phi$ since the $\deg\leq 3$ part of $\Phi$ is uniquely determined.
The authors do not know whether they can construct a non-trivial homomorphism $Y_n\I\C/Y_{n+1}\to \A_{n+k}^c\otimes \Q/A$ for $k\geq 3$ in the same way, where $A$ is some $\Z$-submodule of $\Q$.
\end{remark}

\subsection{Computation of the LMO functor}
This subsection is devoted to the computation of the LMO functor for some bottom-top $q$-tangles up to internal degree $3$, which will be used in the proof of Theorem~\ref{thm:formula}.
We sometimes use identities among bottom-top tangles which fail as bottom-top $q$-tangles, but this difference does not affect the computation of lower-degree terms of the LMO functor due to the next lemma.
Let $P_{u,v,w}$ be the $q$-tangle defined in \cite[Section~5.1]{CHM08},
that is, the identity element $\Id_g$ equipped with the words $(u(vw))$ and $((uv)w)$ at the top and the bottom, respectively.
Here, $u, v, w\in \Mag(\bullet)$ satisfies $g=|u|+|v|+|w|$.

\begin{lemma}
\label{lem:associator}
For any associator,
\[
(\log\Ztilde^Y(P_{u,v,w}))_{\leq 3}=0.
\]
\end{lemma}

\begin{proof}
Set $P_{u,v,w}$ in the form of \cite[Lemma~5.5]{CHM08}.
More precisely, let $w_1=\cdots =w_g=+$ and let $L$ be a disjoint union of $2g$ straight lines in $[-1,1]^3$ endowed with non-associative words $(u(vw)/\bullet\mapsto (+-))$ and $((uv)w/\bullet\mapsto (+-))$ at the top and the bottom, respectively.
As in \cite[Section~3.4]{CHM08},
we have
\[
Z(L)=\Delta_{u',v',w'}^{+++}(\Phi)\in\A(\downarrow^{u'v'w'}),
\]
where $u'=(u/\bullet\mapsto (+-))$, $v'=(v/\bullet\mapsto (+-))$, $w'=(w/\bullet\mapsto (+-))$, respectively.
Let $J$ be a Jacobi diagram appearing in a non-trivial term of $Z(L)$.
Assume that a leg $e$ of $J$ is attached to the $(2i-1)$st line for some $i$.
By the definition of $\Delta_{u',v',w'}^{+++}$,
there also exists another term with opposite sign and with the Jacobi diagram which differs from $J$ only at the point that the leg $e$ is attached to $(2i)$th line.
Hence, $\Delta_{u',v',w'}^{+++}(\Phi)$ vanishes if we connect the top endpoints of the $(2i-1)$st and $(2i)$th lines for all $1\le i\le g$.

Let $\widehat{L}$ be the $1$-manifold consisting of $g$ connected components of the $q$-tangle $P_{u,v,w}$ whose endpoints lie in the bottom $[-1,1]^2\times\{-1\}$.
As we saw above,
it suffices to consider only the terms of $\Ztilde^Y(P_{u,v,w})$ coming from $\deg\geq 1$ parts of exponentials of struts at components of $\widehat{L}$ to which the legs of $J$ attach.
Since $\Phi$ is group-like, $\Phi$ is written as an exponential of an infinite series of connected Jacobi diagrams, and the legs of each diagram are attached to all the three lines.
Hence, we may assume that $\deg J\geq 2$ and the legs of $J$ are also attached to at least three different components of $\widehat{L}$.
Thus, the non-trivial terms of $\Ztilde^Y(P_{u,v,w})-\emptyset$ have $\deg\geq 3+2$.
Therefore, the $\deg\leq 4$ part of $\log\Ztilde^Y(P_{u,v,w})$ is $0$.
Since a connected Jacobi diagram of $\ideg\leq 3$ is of $\deg\leq 4$, the $\ideg\leq 3$ part of $\log\Ztilde^Y(P_{u,v,w})$ is also $0$.
\end{proof}

\begin{figure}[h]
 \centering
 \includegraphics[width=0.7\textwidth]{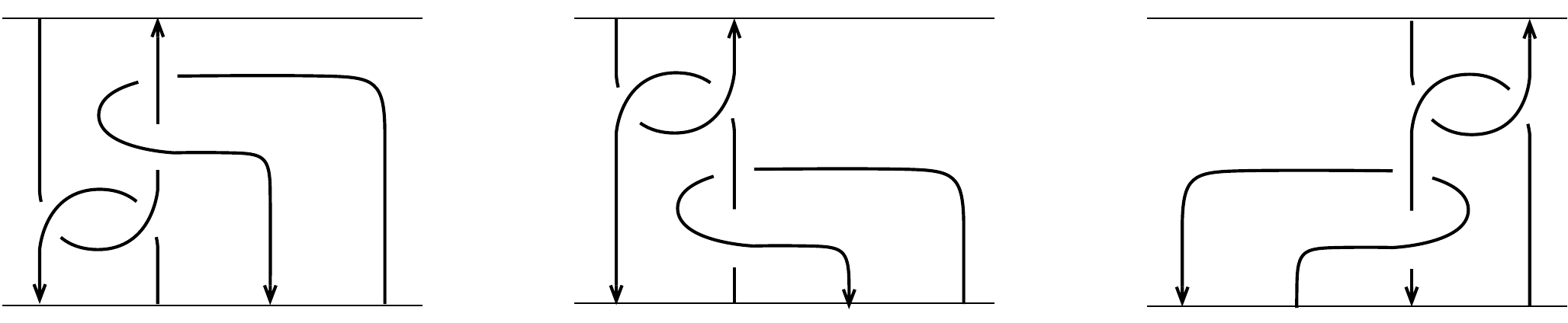}
 \caption{Bottom-top $q$-tangles $\Delta_t$, $\Delta_b$, and ${}_b\Delta$.}
 \label{fig:Deltas}
\end{figure}

Let $\Delta_t$, $\Delta_b$, and ${}_b\Delta$ be bottom-top $q$-tangles in Figure~\ref{fig:Deltas}.
Define $\Delta_t^m$ and $\Delta_b^m$ inductively by $\Delta_t^m=(\Delta_t^{m-1}\otimes \Id_1)\circ \Delta_t$
and $\Delta_b^m=(\Delta_b^{m-1}\otimes \Id_1)\circ \Delta_b$.
For convenience, we also define $\Delta_t^0 = \Delta_b^0 = \Id_1$.

\begin{lemma}
For the bottom-top $q$-tangle ${}_b\Delta$,
\begin{align*}
 (\log\Ztilde({}_b\Delta))_{\leq 2} 
&= \strutgraph{1^+}{2^-} -\strutbottom{1^-}{2^-} +\frac{1}{2}\lambdagraph{0.4}{1^+}{1^-}{2^-} -\frac{1}{12}\Hgraph{1^+}{1^+}{1^-}{2^-} \\
&+\frac{1}{12}\lamlam{1^+}{1^-}{1^-}{2^-} +\frac{1}{8}\Hbottom{1^-}{1^-}{2^-}{2^-} +\frac{1}{8}\dumbbell{1^-}{2^-}.
\end{align*}
\end{lemma}

\begin{proof}
Let $c$ be the bottom-top tangle in \cite[Example~5.2]{CHM08} and recall the notation of bottom-top tangles in Figures~\ref{fig:bt_tangles} and \ref{fig:Deltas}.
By \cite[Table~5.2]{CHM08}, we have
\begin{align*}
(\log\Ztilde(\Id_1\otimes\mu))_{\le2}&= \strutgraph{1^+}{1^-}+\strutgraph{2^+}{2^-}+\strutgraph{3^+}{2^-}-\frac{1}{2}\Ygraph{0.4}{2^+}{3^+}{2^-}+\frac{1}{12}\YYgraph{2^+}{2^+}{3^+}{2^-}+\frac{1}{12}\YYgraphref{2^+}{3^+}{3^+}{2^-}, \\
(\log\Ztilde(c\otimes\Id_1))_{\le2}&= -\strutbottom{1^-}{2^-}+\strutgraph{1^+}{3^-}+\frac{1}{8}\Hbottom{1^-}{1^-}{2^-}{2^-}+\frac{1}{8}\dumbbell{1^-}{2^-}.
\end{align*}
Using the identity ${}_b\Delta=(\Id_1\otimes\mu)\circ(c\otimes\Id_1)$ as bottom-top tangles,
we can compute
$\Ztilde({}_b\Delta)=\Ztilde(\Id_1\otimes\mu)\circ\Ztilde(c\otimes\Id_1)$ and obtain the desired equality.
Alternatively, it can be computed by \cite[Lemma~5.5]{CHM08} directly.
\end{proof}

To prove the formulas for $\ZZ_{n+2}$ and $\zz_{n+2}$ in the next subsection, we here refine \cite[Lemma~4.5]{NSS22GT}.

\begin{lemma}
\label{lem:Delta}
For non-negative integers $m$, the following equalities hold.
\begin{align*}
& (\log\Ztilde(\Delta_t^m))_{\leq 2} 
= \sum_{j=1}^{m+1}\strutgraph{1^+}{j^-} +\sum_{1\leq j<k\leq m+1}\Biggl(-\frac{1}{2}\lambdagraph{0.4}{1^+}{j^-}{k^-} +\frac{1}{4}\Hgraph{1^+}{1^+}{j^-}{k^-} +\frac{1}{12}\lamlamref{1^+}{j^-}{k^-}{k^-}\Biggr) \\
& +\sum_{1\leq j<k<l\leq m+1}\frac{1}{4}\lamlamref{1^+}{j^-}{k^-}{l^-} +\sum_{1\leq j,k<l\leq m+1}\frac{1}{12}\lamlam{1^+}{j^-}{k^-}{l^-}, 
\\
& (\log\Ztilde(\Delta_b^m))_{\leq 2} 
= \strutgraph{1^+}{1^-} \\
&\hspace{-2em} +\sum_{j=2}^{m+1}\Biggl(\strutbottom{1^-}{j^-} -\frac{1}{2}\lambdagraph{0.4}{1^+}{1^-}{j^-} +\frac{1}{12}\Hgraph{1^+}{1^+}{1^-}{j^-} +\frac{1}{12}\lamlamref{1^+}{1^-}{j^-}{j^-} -\frac{1}{8}\Hbottom{1^-}{1^-}{j^-}{j^-} +\frac{1}{8}\dumbbell{1^-}{j^-}\Biggr) \\
&\hspace{-2em} +\sum_{2\leq j<k\leq m+1}\Biggl(-\frac{1}{2}\Ybottom{1^-}{j^-}{k^-} +\frac{1}{4}\lamlamref{1^+}{1^-}{j^-}{k^-} +\frac{1}{12}\lamlam{1^+}{1^-}{j^-}{k^-} -\frac{1}{12}\lamlamref{1^+}{1^-}{k^-}{j^-} +\frac{1}{12}\comb{1^-}{j^-}{j^-}{k^-} \Biggr) \\
&+\sum_{2\leq j<k<l\leq m+1}\frac{1}{4}\comb{1^-}{j^-}{k^-}{l^-} +\sum_{2\leq j<k,l\leq m+1}\frac{1}{12}\Hbottom{1^-}{j^-}{k^-}{l^-}.
\end{align*}
\end{lemma}

\begin{proof}
The case $m=0$ is obvious.
We first show the case $m=1$ using the definitions of $\Delta_t$ and $\Delta_b$.
Recall from the proof of \cite[Lemma~4.5]{NSS22GT} that $\Delta_t=\psi^{-1}\circ\Delta$ and $\Delta_b=(\mu\otimes\Id_1)\circ(\Id_1\otimes c')$, where 
\[
c'=(\mu\otimes\mu)\circ(\Id_1\otimes\Delta_t\otimes\Id_1)\circ(v_+\otimes v_-\otimes v_+).
\]
Then, by \cite[Table~5.2]{CHM08}, we have
\begin{align*}
(\log\Ztilde(\Delta_t))_{\leq 2} 
&=\strutgraph{1^+}{1^-} +\strutgraph{1^+}{2^-} -\frac{1}{2}\lambdagraph{0.4}{1^+}{1^-}{2^-} +\frac{1}{4}\Hgraph{1^+}{1^+}{1^-}{2^-} +\frac{1}{12}\lamlam{1^+}{1^-}{1^-}{2^-} +\frac{1}{12}\lamlamref{1^+}{1^-}{2^-}{2^-} ,\\
(\log\Ztilde(c'))_{\leq 2} &= \strutbottom{1^-}{2^-} -\frac{1}{8}\Hbottom{1^-}{1^-}{2^-}{2^-}+\frac{1}{8}\dumbbell{1^-}{2^-},\\
 (\log\Ztilde(\Delta_b))_{\leq 2} 
&= \strutgraph{1^+}{1^-} 
+\strutbottom{1^-}{2^-} -\frac{1}{2}\lambdagraph{0.4}{1^+}{1^-}{2^-} +\frac{1}{12}\Hgraph{1^+}{1^+}{1^-}{2^-}\\
&+\frac{1}{12}\lamlamref{1^+}{1^-}{2^-}{2^-} -\frac{1}{8}\Hbottom{1^-}{1^-}{2^-}{2^-} +\frac{1}{8}\dumbbell{1^-}{2^-}.
\end{align*}
These complete the proof for $m=1$.
For $m\ge2$, the proof is given by induction on $m$ using $\Delta_t^m=(\Delta_t^{m-1}\otimes \Id_1)\circ \Delta_t$
and $\Delta_b^m=(\Delta_b^{m-1}\otimes \Id_1)\circ \Delta_b$.
\end{proof}

\begin{figure}[h]
 \centering
 \includegraphics[width=0.7\textwidth]{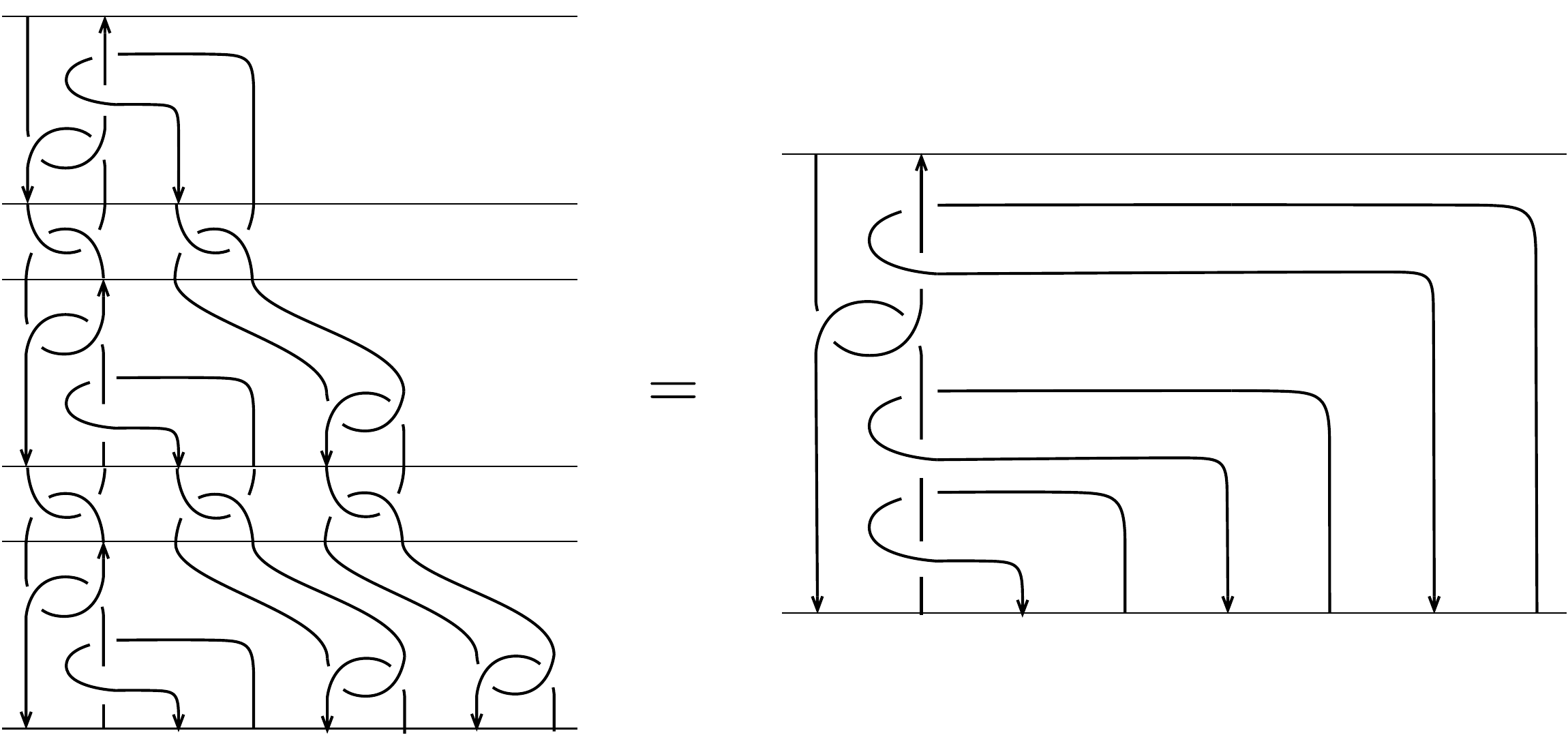}
 \caption{Bottom-top $q$-tangle $(\Delta_b^2\otimes \Id_1)\circ \Delta_t$.}
 \label{fig:tbb}
\end{figure}

The next corollary is directly deduced from Lemma~\ref{lem:Delta} and the equalities
\begin{align*}
3\lamlamref{a}{b}{c}{d} +\lamlam{a}{b}{c}{d} +\lamlam{a}{c}{b}{d} &= 2\lamlamref{a}{b}{c}{d} +2\lamlam{a}{b}{c}{d}, \\
3\lamlamref{a}{b}{c}{d} +\lamlam{a}{b}{c}{d} -\lamlamref{a}{b}{d}{c} &= 2\lamlamref{a}{b}{c}{d} +2\lamlam{a}{b}{c}{d}.
\end{align*}
The result has also been verified by a computer program written in Mathematica.
See Figure~\ref{fig:tbb} for an example of bottom-top $q$-tangles $(\Delta_b^s\otimes \Id_r)\circ\Delta_t^r$.

\begin{corollary}
\label{cor:Delta}
For non-negative integers $r$ and $s$, $(\log\Ztilde((\Delta_b^s\otimes \Id_r)\circ\Delta_t^r))_{\leq 2}$ is equal to
\begin{align*}
& \strutgraph{1^+}{1^-} +\sum_{j=s+2}^{r+s+1}\strutgraph{1^+}{j^-} +\sum_{j=2}^{s+1}\strutbottom{1^-}{j^-} 
 +\sum_{j=s+2}^{r+s+1}\frac{-1}{2}\lambdagraph{0.4}{1^+}{1^-}{j^-} \\
& +\sum_{s+2\leq j<k\leq r+s+1}\frac{-1}{2}\lambdagraph{0.4}{1^+}{j^-}{k^-} 
 +\sum_{j=2}^{s+1} \frac{-1}{2}\lambdagraph{0.4}{1^+}{1^-}{j^-}
 +\sum_{\substack{2\leq j\leq s+1\\ s+2\leq k\leq r+s+1}} \frac{1}{4}\lamlamref{1^+}{1^-}{j^-}{k^-} \\
& +\sum_{k=s+2}^{r+s+1} \biggl(\frac{1}{4}\Hgraph{1^+}{1^+}{1^-}{k^-} +\frac{1}{12}\lamlamref{1^+}{1^-}{k^-}{k^-} +\frac{1}{12}\lamlam{1^+}{1^-}{1^-}{k^-}\biggr) \\
& +\sum_{s+2\leq j<k\leq r+s+1}\biggl(\frac{1}{4}\Hgraph{1^+}{1^+}{j^-}{k^-} +\frac{1}{12}\lamlamref{1^+}{j^-}{k^-}{k^-} +\frac{1}{12}\lamlam{1^+}{j^-}{j^-}{k^-} +\frac{1}{6}\lamlamref{1^+}{1^-}{j^-}{k^-} +\frac{1}{6}\lamlam{1^+}{1^-}{j^-}{k^-} \biggr) \\
& +\sum_{s+2\leq j<k<l\leq r+s+1}\biggl(\frac{1}{6}\lamlamref{1^+}{j^-}{k^-}{l^-} +\frac{1}{6}\lamlam{1^+}{j^-}{k^-}{l^-}\biggr) \\
&+\sum_{j=2}^{s+1}\biggl(\frac{1}{12}\Hgraph{1^+}{1^+}{1^-}{j^-} +\frac{1}{12}\lamlamref{1^+}{1^-}{j^-}{j^-} -\frac{1}{8}\Hbottom{1^-}{1^-}{j^-}{j^-} +\frac{1}{8}\dumbbell{1^-}{j^-}\biggr) \\
&\hspace{-2em} +\sum_{2\leq j<k\leq s+1}\biggl(-\frac{1}{2}\Ybottom{1^-}{j^-}{k^-} 
+\frac{1}{6}\lamlamref{1^+}{1^-}{j^-}{k^-} +\frac{1}{6}\lamlam{1^+}{1^-}{j^-}{k^-} 
+\frac{1}{12}\comb{1^-}{j^-}{j^-}{k^-} +\frac{1}{12}\Hbottom{1^-}{j^-}{k^-}{k^-} \biggr) \\
& +\sum_{2\leq j<k<l\leq s+1}\biggl(\frac{1}{6}\comb{1^-}{j^-}{k^-}{l^-} +\frac{1}{6}\Hbottom{1^-}{j^-}{k^-}{l^-}\biggr).
\end{align*}
\end{corollary}

The next computation is a refinement of \cite[Proposition~5.8]{CHM08}.

\begin{lemma}
\label{lem:Ycob}
For the bottom-top $q$-tangle $Y$,
\begin{align*}
 (\log\Ztilde(Y))_{\leq 3} 
&= -\Ytop{1^+}{2^+}{3^+} +\frac{1}{2}\Htop{1^+}{1^+}{2^+}{3^+} +\frac{1}{2}\Htop{2^+}{2^+}{3^+}{1^+} +\frac{1}{2}\Htop{3^+}{3^+}{1^+}{2^+} \\
&\quad -\frac{1}{6}\HHtop{1^+}{1^+}{1^+}{2^+}{3^+} -\frac{1}{6}\HHtop{2^+}{2^+}{2^+}{3^+}{1^+} -\frac{1}{6}\HHtop{3^+}{3^+}{3^+}{1^+}{2^+} \\
&\quad +\frac{1}{4}\HHtop{1^+}{2^+}{1^+}{2^+}{3^+} +\frac{1}{4}\HHtop{2^+}{3^+}{2^+}{3^+}{1^+} +\frac{1}{4}\HHtop{3^+}{1^+}{3^+}{1^+}{2^+} -2\buYtop{1^+}{2^+}{3^+}.
\end{align*}
\end{lemma}

\begin{proof}
We first recall that $(\log\Ztilde(Y))_{\leq 2}$ is determined in \cite[Table~5.2]{CHM08}.
By the identities 
\[
Y\circ(\eta\otimes\Id_2) = Y\circ(\Id_1\otimes\eta\otimes\Id_1) = Y\circ(\Id_2\otimes\eta) = \varepsilon\otimes\varepsilon
\]
as bottom-top tangles in \cite[Proof of Proposition~5.8]{CHM08}, each diagram in $(\log\Ztilde(Y))_{\leq 3}$ should have all $1^+$, $2^+$ and $3^+$.
We may assume $(\log\Ztilde(Y))_3$ is a linear sum of tree Jacobi diagrams of $\ideg=3$ and the 1-loop Jacobi diagram $O(1^+,2^+,3^+)$.
By the AS and STU relations, we can write $(\log\Ztilde(Y))_{3}$ of the form
\begin{align*}
& (\log\Ztilde(Y))_{3} = a_1\HHtop{1^+}{1^+}{1^+}{2^+}{3^+} +a_2\HHtop{2^+}{2^+}{2^+}{3^+}{1^+} +a_3\HHtop{3^+}{3^+}{3^+}{1^+}{2^+} \\
&\quad +a_4\HHtop{1^+}{2^+}{1^+}{2^+}{3^+} +a_5\HHtop{2^+}{3^+}{2^+}{3^+}{1^+} +a_6\HHtop{3^+}{1^+}{3^+}{1^+}{2^+} +a_7\buYtop{1^+}{2^+}{3^+}
\end{align*}
for some $a_i\in \Q$.
As in \cite[Section~5.1]{CHM08}, for $p,q\in\Z_{>0}$, let $\psi_{p,q}$ be the bottom-top tangle which represents the braiding of the monoidal category $\LCob_q$.
Explicitly, $\psi_{2,1}$ is given by $\psi_{2,1}=(\psi_{1,1}\otimes\Id_1)\circ (\Id_1\otimes \psi_{1,1})$.
Thus, we have
\[
(\log\Ztilde(\psi_{2,1}))_{\le2}=
\strutgraph{1^+}{2^-}+\strutgraph{2^+}{3^-} +\strutgraph{3^+}{1^-} 
-\frac{1}{2}\Hgraph{1^+}{3^+}{1^-}{2^-}-\frac{1}{2}\Hgraph{2^+}{3^+}{1^-}{3^-}.
\]
By \cite[Table~5.2]{CHM08}, we also have
\[
(\log\Ztilde(\Id_2\otimes S^2))_{\le2}=
\strutgraph{1^+}{1^-}+\strutgraph{2^+}{2^-} +\strutgraph{3^+}{3^-} +\frac{1}{2}\ \phigraph{3^+}{3^-} -\frac{1}{2}\Hgraph{3^+}{3^+}{3^-}{3^-}.
\]
From the identity
\[
Y\circ\psi_{2,1}\circ (\Id_2\otimes S^{2})=Y
\]
as bottom-top tangles,
we have $a_1=a_2=a_3$ and $a_4=a_5=a_6$.

To determine $a_1,a_4,a_7$, we focus on two bottom-top tangles $M_1$ and $M_2$ drawn in Figure~\ref{fig:Mas07}, which are equivalent due to \cite[Figure~4]{Mas07}.
Let us compare the values of the LMO functor.
As in Figure~\ref{fig:Mas07}, $M_1$ decompose as $(\Id_1\otimes Y\otimes\Id_1)\circ M_1'$, where
\[
M_1'=
(\Id_3\otimes\mu\otimes\Id_1)\circ(\Id_4\otimes v_+\otimes\Id_1)\circ(\Id_1\otimes\psi_{1,1}\otimes\Id_2)\circ(\Delta_b\otimes\Id_1\otimes{}_b\Delta)\circ(\Delta_t\otimes\Id_1).
\]
Then, $(\log\Ztilde(M_1'))_{\leq 2}$ is equal to
\begin{align*}
& \strutgraph{1^+}{1^-} +\strutgraph{1^+}{2^-} +\strutgraph{2^+}{5^-} +\strutbottom{1^-}{3^-} -\frac{1}{2}\strutbottom{4^-}{4^-} -\strutbottom{4^-}{5^-} \\
& -\frac{1}{2}\Ygraphref{0.4}{1^+}{1^-}{2^-} -\frac{1}{2}\Ygraphref{0.4}{1^+}{1^-}{3^-} +\frac{1}{2}\Ygraphref{0.4}{2^+}{4^-}{5^-} +\frac{1}{8}\dumbbell{1^-}{3^-} +\frac{1}{48}\dumbbell{4^-}{4^-} +\frac{5}{24}\dumbbell{4^-}{5^-} \\
& +\frac{1}{4}\Hgraph{1^+}{1^+}{1^-}{2^-} +\frac{1}{12}\Hgraph{1^+}{1^+}{1^-}{3^-} +\frac{1}{12}\lamlamref{1^+}{1^-}{2^-}{2^-} -\frac{1}{4}\lamlamref{1^+}{1^-}{3^-}{2^-} -\frac{1}{8}\Hbottom{1^-}{1^-}{3^-}{3^-} \\
& +\frac{1}{12}\lamlamref{1^+}{1^-}{3^-}{3^-} +\frac{1}{12}\lamlam{1^+}{1^-}{1^-}{2^-} +\frac{1}{12}\lamlam{2^+}{4^-}{4^-}{5^-} -\frac{1}{12}\Hgraph{2^+}{2^+}{4^-}{5^-} +\frac{1}{6}\Hbottom{4^-}{4^-}{5^-}{5^-}.
\end{align*}

Using $(\log\Ztilde(Y))_{\leq 3}$ and $(\log\Ztilde(M_1'))_{\leq 2}$, we compute $(\log\Ztilde(M_1))_{\leq 3}$ as follows:
\begin{align*}
& \strutgraph{1^+}{1^-} +\strutgraph{2^+}{2^-} -\frac{1}{2}\phigraph{1^+}{1^-} -\Ygraphref{0.4}{1^+}{1^-}{2^-} +\frac{1}{2}\Hgraph{1^+}{1^+}{1^-}{1^-} +\frac{1}{2}\lamlamref{1^+}{1^-}{2^-}{2^-} +\frac{1}{2}\Hgraph{1^+}{2^+}{1^-}{2^-} \\
& +\left(-\frac{5}{12}+5a_1-a_7\right)\buYref{0.3}{1^+}{1^-}{2^-} -\left(a_1+\frac{1}{6}\right)T(1^-,1^+,1^+,1^+,2^-) \\
& -\left(a_4+\frac{3}{4}\right)T(1^-,1^+,1^+,2^-,1^-) -\left(a_1+\frac{1}{6}\right)T(1^+,1^-,1^-,1^-,2^-) \\
& +\left(\frac{1}{4}-a_4\right)T(1^+,1^-,2^-,2^-,1^+) +\left(a_4-\frac{1}{4}\right)T(1^-,1^+,2^-,2^-,1^-) \\
& +\frac{1}{12}T(1^+,1^-,2^+,2^+,2^-) -\frac{1}{4}T(1^-,2^+,2^-,2^-,1^+) +\frac{1}{4}T(1^-,2^-,2^-,2^+,1^+) \\
& -a_1T(1^-,2^-,2^-,2^-,1^+),
\end{align*}
where
\[
T(a_1,a_2,\ldots, a_n)=
\begin{tikzpicture}[baseline=1.3ex, scale=0.25, dash pattern={on 2pt off 1pt}]
\node [left] at (0,0){$a_1$};
\node [above] at (2,2){$a_2$};
\node [below] at (4,2){$\cdots$};
\node [below] at (8,2){$\cdots$};
\node [above] at (10,2){$a_{n-1}$};
\node [right] at (12,0){$a_n$};
\draw (0,0) -- (12,0);
\draw (2,0) -- (2,2);
\draw (10,0) -- (10,2);
\end{tikzpicture}\ ,
\]
for $a_1,a_2,\ldots, a_n \in \{1^{\pm},2^{\pm},\ldots,g^{\pm}\}$.
On the other hand, $M_2$ decompose as $M_2'\circ M_2''$ as in Figure~\ref{fig:Mas07}.
Note that $M_2$ is the same as $M_1$ in \cite[Proposition~A.1]{NSS22GT}.
By \cite[Lemma~5.5]{CHM08} and computing the product by a computer program, we obtain $(\log\Ztilde(M_2'))_{\leq 3}$ as follows:
\begin{align*}
& \strutgraph{1^+}{1^-} +\strutgraph{2^+}{2^-} +\frac{1}{2}\Hgraph{1^+}{1^+}{1^-}{1^-} -\frac{1}{2}\ \phigraph{1^+}{1^-}\ .
\end{align*}
Similarly, $(\log\Ztilde(M_2''))_{\leq 3}$ is equal to
\begin{align*}
&\strutgraph{1^+}{1^-} +\strutgraph{2^+}{2^-} -\Ygraphref{0.4}{1^+}{1^-}{2^-} +\frac{1}{2}\Hgraph{1^+}{2^+}{1^-}{2^-} +\frac{1}{2}\lamlamref{1^+}{1^-}{2^-}{2^-} -\frac{1}{4}\buYref{0.3}{1^+}{1^-}{2^-} +\frac{1}{24}T(1^+,1^-,2^-,2^-,1^+) \\
&+\frac{1}{4}T(1^+,1^-,2^-,2^+,2^-) +\frac{1}{12}T(1^+,1^-,2^+,2^+,2^-) -\frac{1}{24}T(2^-,1^-,1^+,2^-,1^+) \\
&-\frac{1}{6}T(2^-,1^-,2^-,2^-,1^+) +\frac{1}{2}T(2^-,1^-,2^+,2^-,1^+) -\frac{1}{2}T(2^-,1^-,2^-,2^+,1^+).
\end{align*}
Then, we can compute $(\log\Ztilde(M_2))_{\leq 3}$ as follows:
\begin{align*}
&\strutgraph{1^+}{1^-} +\strutgraph{2^+}{2^-} -\Ygraphref{0.4}{1^+}{1^-}{2^-} +\frac{1}{2}\Hgraph{1^+}{2^+}{1^-}{2^-} +\frac{1}{2}\lamlamref{1^+}{1^-}{2^-}{2^-} +\frac{1}{2}\Hgraph{1^+}{1^+}{1^-}{1^-} -\frac{1}{2}\ \phigraph{1^+}{1^-} \\
&+\frac{3}{4}\buYref{0.3}{1^+}{1^-}{2^-} +T(1^+,1^-,1^+,2^-,1^-) +\frac{1}{12}T(1^+,1^-,2^+,2^+,2^-) \\
& -\frac{1}{4}T(1^-,2^+,2^-,2^-,1^+) +\frac{1}{4}T(1^-,2^-,2^-,2^+,1^+) -\frac{1}{6}T(2^-,1^-,2^-,2^-,1^+).
\end{align*}
Comparing $(\log\Ztilde(M_1))_{\leq 3}$ and $(\log\Ztilde(M_2))_{\leq 3}$, we have 
\[
-\frac{5}{12}+5a_1-a_7=\frac{3}{4},\ 
a_1+\frac{1}{6}=0,\ 
a_4+\frac{3}{4}=1,\ 
\frac{1}{4}-a_4=0,\ 
a_1=-\frac{1}{6},
\]
and thus $a_1=-1/6$, $a_4=1/4$, $a_7=-2$.
\end{proof}

\begin{figure}[h]
 \centering
 \includegraphics[width=\textwidth]{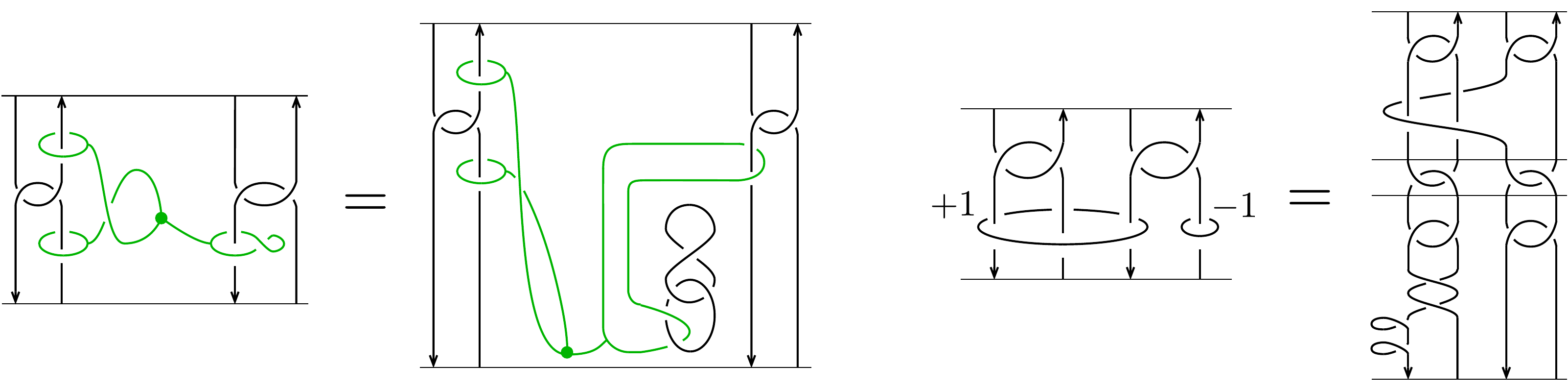}
 \caption{Two equivalent bottom-top tangles $M_1$ and $M_2=M_2'\circ M_2''$.}
 \label{fig:Mas07}
\end{figure}

\begin{remark}
Lemma~\ref{lem:Ycob} will be used in the proof of Theorem~\ref{thm:formula}.
It is worth noting that in the proof of the existence of $3$-torsion in Theorem~\ref{thm:Y6/Y7} we use a part of the formula, which is independent of Lemma~\ref{lem:Ycob}.
In fact, the homomorphism $\zz_{8,4}\circ\ss_6\colon \A^c_{6,2}\to \A^c_{8,4}\otimes\Q/\tfrac{1}{2}\Z$ used in the proof of Theorem~\ref{thm:Y6/Y7} does not depend on the $\ideg\geq 3$ part of $\log\Ztilde(Y)$.
Therefore, the existence of torsion elements of order $3$ in $Y_6\I\C/Y_7$ is shown without computer. 
\end{remark}

\subsection{Formulas of our invariants}
Recall that $\Ztilde^Y_n(\S(J))=(-1)^{n+b_0(J)+e}J$ holds for a Jacobi diagram $J\in \A^Y_n$, where $e$ is the number of internal edges of $J$ and $b_k$ denotes the $k$th Betti number (see the end of the proof of \cite[Theorem~7.11]{CHM08}).
One can easily check that $(-1)^{n+b_0(J)+e}=(-1)^{b_1(J)}$.
Let $U(J)$ denote the set of univalent vertices of $J$.
In this subsection, a pair $\{u,v\}$ for $u,v\in U(J)$ is called a \emph{leaf pair} if they are adjacent to a common vertex.
For a Jacobi diagram $J$, let $U^\pm$ denote the subset of univalent vertices colored by $i^\pm$ for some $i$, respectively.
We have $U^+\sqcup U^- =U(J)$.
Let $e(v)$ denote the edge incident to a univalent vertex $v$.

Let $J$ be a Jacobi diagram of $\ideg=n$ and, for each color $c\in\{1^\pm,\ldots,g^\pm\}$, fix a total order $\prec$ on the set of univalent vertices of $J$ colored by $c$.
We then define $\delta_0(J)$, $\delta_1(J)$, and $\delta_2(J) \in \A^Y_{n+2}\otimes\Q/\tfrac{1}{2}\Z$ by
\begin{align*}
\delta_0(J)&=
\sum_{\{u,v\}} \frac{1}{4}\delta^{Y}_{u}(\delta^{Y}_{v}(J)) 
+\sum_{v\in U^+}\Bigl(\frac{1}{4}\delta^{+}_{v}(J) +\frac{1}{12}\delta^{-}_{v}(J)\Bigr)
+\sum_{v\in U^-}\frac{1}{12}\delta^{+}_{v}(J) \\
&
+\sum_{u\in U(J),\, v\in U(\delta_u^{\shortmid\shortmid}(J))} \frac{1}{4}\delta_v^Y(\delta_u^{\shortmid\shortmid}(J)) 
+\sum_{\{u,v\}} \frac{1}{4}\delta^{\shortmid\shortmid}_{u}(\delta^{\shortmid\shortmid}_{v}(J))
+\sum_{v\in U}\frac{1}{6}\delta^{\shortmid\shortmid\shortmid}_{v}(J), 
\\
\delta_1(J)&=
\sum_{\substack{u,v,w\in U(J)\\ u\prec v}} \frac{1}{4}\lambda_{u,v}(\delta^{Y}_{w}(J)) 
+\sum_{u\prec v\in U^+} \frac{1}{4}H_{u,v}(J)
+\sum_{u\prec v\in U(J)} \frac{1}{6}H'_{u,v}(J)
\\
&+\sum_{\substack{\{u,v\}\\ \ell(u)=\ell(v)^\ast}} \frac{1}{4}H_{u,v}(J)
+\sum_{v\in U^-} \frac{1}{8}\beta_{e(v)}(J) 
+\sum_{\substack{u\in U(J),\, v,w\in U(\delta_u(J))\\ U(J)\cap \{v,w\}\neq \emptyset}}\frac{1}{4}\lambda_{v,w}(\delta^{\shortmid\shortmid}_{u}(J)), 
\\
\delta_2(J)&=
\sum_{\substack{\{\{u,v\},\{u',v'\}\}\\ u\prec v,\, u'\prec v'}} \frac{1}{4}\lambda_{u,v}(\lambda_{u',v'}(J)) 
+\sum_{u\prec v\prec w\in U} \frac{1}{6}\lambda_{u,v,w}(J), 
\end{align*}
where $u,u',v,v',w$ are distinct in each summation and $\{u,v\}$ runs over \textup{(}unordered\textup{)} pairs of univalent vertices of $J$.
Here the operations above are defined by
\begin{gather*}
\delta^Y_v\Bigl(\strutgraph{\ell(v)}{}\Bigr)=\Ygraph{0.4}{\ell(v)}{\ell(v)^\ast}{},
\quad
\delta^{+}_v(\yokostrut{\ell(v)}\ )=
\begin{tikzpicture}[scale=0.4, baseline={(0,0)}, densely dashed]
  \draw (0,0) -- (0,1) node[anchor=south] {\Small $\ell(v)^+$};
  \draw (2,0) -- (2,1) node[anchor=south] {\Small $\ell(v)^+$};
  \draw (0,0) -- (3,0);
  \draw (0,0) -- (0,-1) node[anchor=north] {\Small $\ell(v)^-$};
\end{tikzpicture},
\quad
\delta^{-}_v(\yokostrut{\ell(v)}\ )=
\begin{tikzpicture}[scale=0.4, baseline={(0,0)}, densely dashed]
  \draw (0,0) -- (0,1) node[anchor=south] {\Small $\ell(v)^+$};
  \draw (0,0) -- (3,0);
  \draw (0,0) -- (0,-1) node[anchor=north] {\Small $\ell(v)^-$};
  \draw (2,0) -- (2,-1) node[anchor=north] {\Small $\ell(v)^-$};
\end{tikzpicture},
\\
\delta^{\shortmid\shortmid}_{v}(J)\Bigl(\ \Tref{\ell(v)}\ \Bigr)=
\begin{tikzpicture}[scale=0.4, baseline={(0,0.4)}, densely dashed]
  \draw (0,0) -- (0,2) node[anchor=south] {\Small $\ell(v)$};
  \draw (2,0) -- (2,2) node[anchor=south] {\Small $\ell(v)$};
  \draw (-1,0) -- (3,0);
\end{tikzpicture}\ ,
\quad
\delta^{\shortmid\shortmid\shortmid}_{v}\Bigl(\ \Tref{\ell(v)}\ \Bigr)=
\begin{tikzpicture}[scale=0.4, baseline={(0,0.4)}, densely dashed]
  \draw (0,0) -- (0,2) node[anchor=south] {\Small $\ell(v)$};
  \draw (2,0) -- (2,2) node[anchor=south] {\Small $\ell(v)$};
  \draw (4,0) -- (4,2) node[anchor=south] {\Small $\ell(v)$};
  \draw (-1,0) -- (5,0);
\end{tikzpicture}\ ,
\\
\lambda_{u,v}\Bigl(\strutt{\ell(u)}{\ell(v)}\Bigr)=\Ygraphref{0.4}{\ell(v)}{}{}\ ,
\quad
\lambda_{u,v,w} \Bigl(\struttt{\ell(u)}{\ell(v)}{\ell(w)}\Bigr)=\lamlam{\ell(v)}{}{}{}\ +\lamlamref{\ell(v)}{}{}{}\ ,
\\
H_{u,v}\Bigl(\strutt{\ell(u)}{\ell(v)}\Bigr)=\Hgraph{\ell(u)}{\ell(v)}{}{},
\quad
H'_{u,v}\Bigl(\strutt{\ell(u)}{\ell(v)}\Bigr)=\Igraph{\ell(u)^-}{\ell(v)^+} +\Hgraph{\ell(u)^-}{\ell(v)^+}{}{},
\end{gather*}
where $(i^{\pm})^\varepsilon$ is defined to be $i^\varepsilon$ for $\varepsilon\in \{\pm 1\}$ in $\delta^{+}_{v}$, $\delta^{-}_{v}$, and $H'_{u,v}$.
Furthermore, $\beta_e(J)$ is defined by replacing an edge $e$ by $\,\phigraph{}{}\,$.
Note that $\delta_k(J)$ increases $b_1(J)$ by $k$ when $J$ is connected.

We do not use the next proposition, but it is worth stating here.

\begin{proposition}
The elements $\delta_0(J)$, $\delta_1(J)$, and $\delta_2(J)$ give rise to well-defined homomorphisms $\delta_0,\delta_1,\delta_2\colon \A^Y_{n}\to \A^Y_{n+2}\otimes\Q/\tfrac{1}{2}\Z$.
More precisely, the terms
\begin{align*}
\frac{1}{4}\delta^{Y}_{u}(\delta^{Y}_{v}(J)),\ 
\frac{1}{4}\delta^{+}_{v}(J),\ 
\frac{1}{12}\delta^{-}_{v}(J),\ 
\frac{1}{12}\delta^{+}_{v}(J),\ 
\frac{1}{4}\delta_v^Y(\delta_u^{\shortmid\shortmid}(J)),\ 
\frac{1}{4}\delta^{\shortmid\shortmid}_{u}(\delta^{\shortmid\shortmid}_{v}(J)),\ 
\frac{1}{6}\delta^{\shortmid\shortmid\shortmid}_{v}(J) 
\end{align*}
in $\delta_0(J)$,
\begin{align*}
\frac{1}{4}\lambda_{u,v}(\delta^{Y}_{w}(J)),\ 
\frac{1}{4}H_{u,v}(J),\ 
\frac{1}{6}H'_{u,v}(J),\ 
\frac{1}{4}H_{u,v}(J),\ 
\frac{1}{8}\beta_{e(v)}(J),\ 
\frac{1}{4}\lambda_{v,w}(\delta^{\shortmid\shortmid}_{u}(J)) 
\end{align*}
in $\delta_1(J)$, and
\begin{align*}
\frac{1}{4}\lambda_{u,v}(\lambda_{u',v'}(J)),\ 
\frac{1}{6}\lambda_{u,v,w}(J), 
\end{align*}
in $\delta_2(J)$ are invariant under the AS, IHX, and self-loop relations and independent of the total order $\prec$.
\end{proposition}

\begin{proof}
By the AS relation and the equality $-\frac{1}{4}=\frac{1}{4}\in\Q/\frac{1}{2}\Z$, the terms
$\frac{1}{4}\delta_v^Y(\delta_u^{\shortmid\shortmid}(J))$,
$\frac{1}{4}\delta^{\shortmid\shortmid}_{u}(\delta^{\shortmid\shortmid}_{v}(J))$,
$\frac{1}{6}\delta^{\shortmid\shortmid\shortmid}_{v}(J)$,
$\frac{1}{4}H_{u,v}(J)$,
$\frac{1}{4}\lambda_{u,v}(\lambda_{u',v'}(J))$,
and $\frac{1}{4}\lambda_{v,w}(\delta^{\shortmid\shortmid}_{u}(J))$ are well-defined.
Noting that $u\prec v\in U(J)$ implies that $\ell(u)=\ell(v)$, we also see that $\frac{1}{4}\lambda_{u,v}(\delta^{Y}_{w}(J))$ is well-defined.

By the IHX relation and the equality $\frac{1}{3}=-\frac{1}{6}\in\Q/\frac{1}{2}\Z$, we have
\[
\frac{1}{6}H'_{u,v}\Bigl(\strutt{\ell(u)}{\ell(v)}\Bigr)
= \frac{1}{6}\Igraph{\ell(u)^-}{\ell(v)^+} +\frac{1}{6}\Hgraph{\ell(u)^-}{\ell(v)^+}{}{}
= -\frac{1}{6}\Hgraph{\ell(v)^+}{\ell(u)^-}{}{} -\frac{1}{6}\Hgraph{\ell(u)^-}{\ell(v)^+}{}{}
\]
in $\A_{n+2}^Y\otimes\Q/\frac{1}{2}\Z$.
Since $\ell(u)=\ell(v)$ when $u\prec v\in U(J)$, we have $\frac{1}{6}H'_{u,v}(J)=\frac{1}{6}H'_{v,u}(J)$, and $\frac{1}{6}H'_{u,v}(J)$ is well-defined.
In a similar way, we see that $\frac{1}{6}\lambda_{u,v,w}(J)$ does not depend on the choice of a total order
and is well-defined.
The rest of the terms $\frac{1}{4}\delta^{Y}_{u}(\delta^{Y}_{v}(J))$,
$\frac{1}{4}\delta^{+}_{v}(J)$,
$\frac{1}{12}\delta^{-}_{v}(J)$,
$\frac{1}{12}\delta^{+}_{v}(J)$,
and $\frac{1}{8}\beta_{e(v)}(J)$ 
are apparently well-defined.
\end{proof}

\begin{example}
Let $J=T(1^+,2^+,2^-,1^+)$.
Then, $\delta_2(J)=0$ and
\begin{align*}
\delta_1(J)
&=\frac{1}{4}O(1^+,2^+,2^-,2^+) +\frac{1}{4}O(1^+,2^-,2^+,2^-) +\frac{1}{6}O(1^+,1^-,2^+,2^-) \\
&+\frac{1}{3}O(1^+,1^-,2^-,2^+) +\frac{1}{4}O(1^+,1^+,2^+,2^-) +\frac{1}{4}O(1^+,2^+,1^+,2^-), \\
\delta_0(J)
&=\frac{1}{4}T(2^-,1^-,1^+,1^+,1^-,2^+) +\frac{1}{12}T(2^-,1^-,1^-,1^+,1^+,2^+) +\frac{1}{12}T(2^-,1^-,1^+,1^-,1^+,2^+) \\
&+\frac{1}{4}T(1^-,1^+,1^+,1^+,2^+,2^-) +\frac{1}{12}T(2^-,1^+,1^-,1^+,1^-,2^+) +\frac{1}{12}T(2^-,1^+,1^+,1^-,1^-,2^+) \\
&-\frac{1}{12}T(2^-,1^+,1^-,1^-,1^+,2^+) +\frac{1}{12}T(2^-,1^+,1^+,1^+,1^+,2^+) +\frac{1}{6}T(1^+,2^-,2^-,2^+,2^-,1^+) \\
&+\frac{1}{4}T(1^+,2^-,2^+,2^+,2^-,1^+) +\frac{1}{6}T(1^+,2^+,2^-,2^+,2^+,1^+).
\end{align*}
\end{example}

Now, we can show our main result in this paper.

\begin{theorem}
\label{thm:formula}
Let $J \in \A^c_{n}$ and, for each color $c\in\{1^\pm,\ldots,g^\pm\}$, fix a total order $\prec$ on the set of univalent vertices of $J$ colored by $c$.
Then, 
\[
(-1)^{b_1(J)+1}\zz_{n+2}(\ss_n(J)) = \delta_0(J)+\delta_1(J)+\delta_2(J) \in \A^c_{n+2}\otimes\Q/\tfrac{1}{2}\Z.
\]

Moreover, for $J \in \A^Y_{n}$, 
\begin{align*}
(-1)^{b_1(J)}\ZZ_{n+2}(\S_n(J))
&= \delta_0(J)+\delta_1(J)+\delta_2(J) \\
&\quad 
+\sum_{Y}\biggl(
\frac{1}{4}\delta^Y(J\sqcup Y)
+\frac{1}{4}\lambda(J\sqcup Y)
+\frac{1}{3!}J\sqcup Y^{\sqcup 2}
+\frac{1}{4}\delta^{\shortmid\shortmid}(J\sqcup Y)
\biggr),
\end{align*}
in $\A^Y_{n+2}\otimes\Q/\tfrac{1}{2}\Z$, where $Y$ runs over connected components of $J$ such that $\ideg(Y)=1$, and
\[
\delta^{Y}(J)=\sum_{v\in U(J)} \delta^{Y}_{v}(J),\
\lambda(J)=\sum_{u\prec v\in U(J)} \lambda_{u,v}(J),\
\delta^{\shortmid\shortmid}(J)=\sum_{v\in U(J)} \delta^{\shortmid\shortmid}_{v}(J).
\]
\end{theorem}

\begin{proof}
The following argument is a refinement of the proof of \cite[Theorem~1.1]{NSS22GT}.
If we prove the formula for $\ZZ_{n+2}$, then that for $\zz_{n+2}$ is a direct consequence.
Indeed, for a connected Jacobi diagram $J \in \A^c_n$, if $n\geq 2$, we would have
\begin{align*}
(-1)^{b_1(J)+1}\zz_{n+2}(\ss(J)) 
&= (-1)^{b_1(J)+1}\Ztilde^Y_{n+2}(\ss(J)) \\
&= (-1)^{b_1(J)}\Ztilde^Y_{n+2}(\S(J)) \\
&= \delta_0(J)+\delta_1(J)+\delta_2(J),
\end{align*}
where the first equality comes from Remark~\ref{rem:logZ} and the last one from the fact that $J$ has no connected component of $\ideg=1$.
In the case $n=1$, the formula for $\ZZ_{3}$ gives
\begin{align*}
(-1)^{0+1}\zz_{1+2}(\ss(J)) 
&= -\Ztilde^Y_{3}(\ss(J)) -\frac{1}{3}J^{\sqcup 3} \\
&= (-1)^{0}\Ztilde^Y_{3}(\S(J))+\frac{1}{3!}J^{\sqcup 3} \\
&= \delta_0(J)+\delta_1(J)+\delta_2(J)
\end{align*}
since $-\frac{1}{3}=\frac{1}{6}$ in $\Q/\tfrac{1}{2}\Z$ and $\frac{1}{4}\delta^Y(J\sqcup J)=0$, $\frac{1}{4}\delta^{\shortmid\shortmid}(J\sqcup J)=0$ in $\A^Y_{3}\otimes\Q/\tfrac{1}{2}\Z$.

Let us prove the formula for $\ZZ_{n+2}$.
Let $J \in \A^Y_n$ be a Jacobi diagram and we draw $J$ as in Figure~\ref{fig:Jacobi} according to $\prec$.
Let $e_{g+1},e_{g+2},\dots,e_{g+3n}$ denote the half-edges incident to the trivalent vertices of $J$.
Let $N=\{g+1,g+2,\dots,g+3n\}$.
Define $V$, $E$, $L^t_i$ and $L^b_i$ for $i=1,\dots,g$ by
\begin{align*}
V&= \left\{(j,k,l)\in N^3\Bigm| \parbox{17em}{$e_j$, $e_k$, and $e_l$ are the three half-edges incident to a trivalent vertex clockwise}\right\}\Big/\text{cyclic permutation}, \\
E&= \left\{(j,k)\in N^2\Bigm| \parbox{17em}{$e_j$ and $e_k$ are the two half-edges of an edge connecting two trivalent vertices}\right\}\Big/\text{permutation}, \\
L^t_i&= \{j\in N\mid \text{the univalent vertex of the edge containing $e_j$ is colored with $i^+$}\}, \\
L^b_i&= \{j\in N\mid \text{the univalent vertex of the edge containing $e_j$ is colored with $i^-$}\}.
\end{align*}
Let $r_i= \# L^t_i$ and $s_i= \# L^b_i$.
For $j,k \in L^t_i$ (or $j,k \in L^b_i$), we write $j\prec k$ if $v(e_j)\prec v(e_k)$, where $v(e_j)$ is the univalent vertex incident to the edge containing the half-edge $e_j$.

\begin{figure}[h]
 \centering
 \includegraphics[width=0.6\textwidth]{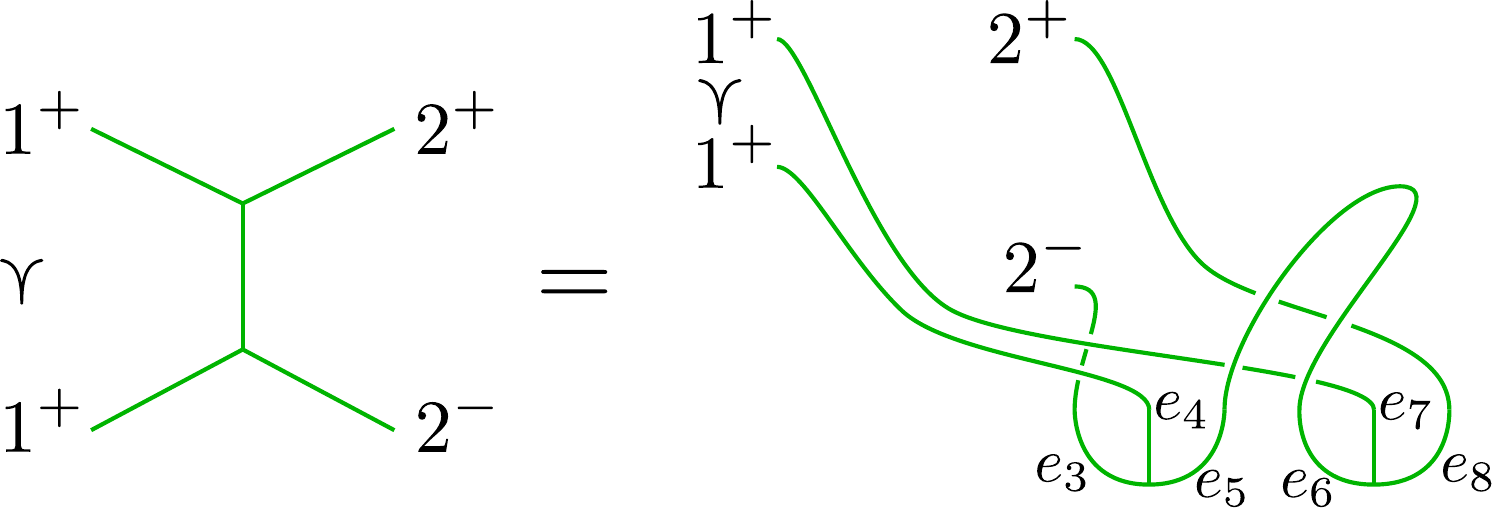}
 \caption{A specific drawing of a Jacobi diagram with $V=\{(3,4,5),(6,7,8)\}$, $E=\{(5,6)\}$, $L^t_1=\{4,7\}$, and $4\prec 7$.}
 \label{fig:Jacobi}
\end{figure}

\begin{figure}[h]
 \centering
 \includegraphics[width=0.9\textwidth]{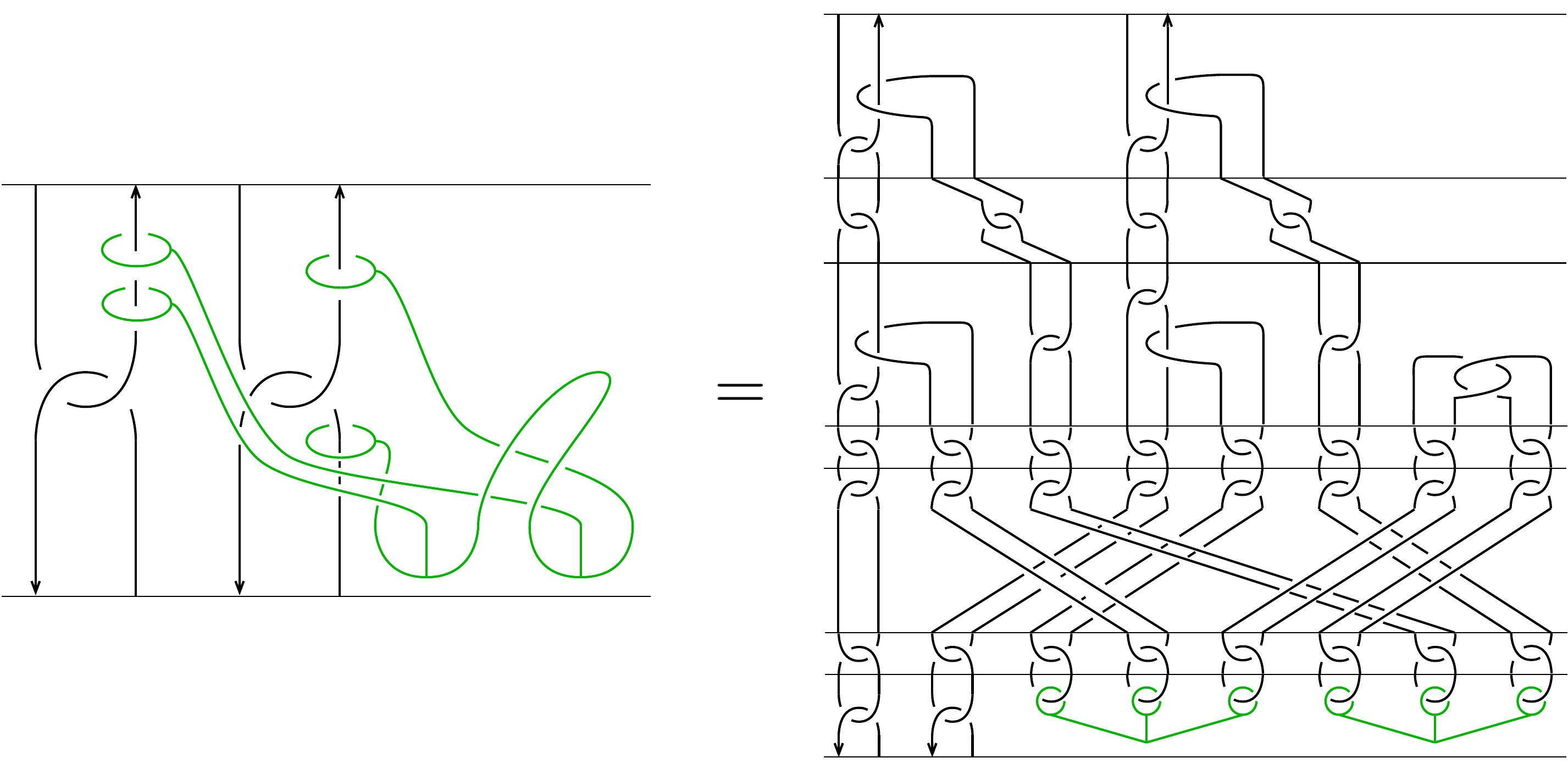}
 \caption{The decomposition corresponding to Figure~\ref{fig:Jacobi}.}
 \label{fig:decomposition}
\end{figure}

Let $G$ be a graph clasper realizing $J$.
By the well-definedness of $\S$, we may assume that $G$ is obtained from a specific drawing of $J$ as in Figure~\ref{fig:Jacobi}.
Then the corresponding bottom-top tangle $([-1,1]^3,\gamma_g)_G$ decomposes as
\[
(\Id_g\otimes Y^{\otimes n})\circ \Psi\circ \left(\biggl(\bigotimes_{i=1}^{g}((\Delta_b^{s_i}\otimes\Id_{r_i})\circ \Delta_t^{r_i})\biggr)\otimes c^{\otimes\# E}\right),
\]
where $\Psi$ consists of $\psi_{1,1}^{\pm 1}$, $P_{u,v,w}^{\pm 1}$, and $\Id_m$ in \cite{CHM08}.
See Figure~\ref{fig:decomposition} for an example of the decomposition.
We write $\gamma$ for the third factor of the decomposition.

By the definition of $\S$ and \cite[Proof of Theorem~7.11]{CHM08}, one has
\[
\Ztilde(\S(J)) = \sum_{G'\subset G}(-1)^{|G'|}\Ztilde((\Sigma_{g,1}\times[-1,1])_{G'})
 = (-1)^{n+|G|}\left(\Id_g\otimes(\emptyset - \Ztilde(Y))^{\otimes n}\right)\circ \Ztilde(\Psi\circ \gamma).
\]
It follows from $\ideg(\emptyset - \Ztilde(Y))\geq 1$ that
\[
(-1)^{b_1(J)}\Ztilde_{n+2}(\S(J)) = (-1)^{\# E}\sum_{d=0}^{2}\left(\Id_g\otimes(\emptyset - \Ztilde(Y))^{\otimes n}_{n+d}\right)\circ \Ztilde_{2-d}(\Psi\circ \gamma).
\]

Since $(\log\Ztilde(\Psi))_{\leq 2}$ is a sum of $H$-graphs with coefficients $\pm\frac{1}{2}$ and struts, the composition of $(\log\Ztilde(\Psi))_2$ and struts with integral coefficients is zero in $\A^c_{2}\otimes\Q/\tfrac{1}{2}\Z$.
Now, in $\A^c_{\leq 1}\otimes\Q\oplus\A^c_{2}\otimes\Q/\tfrac{1}{2}\Z$, Corollary~\ref{cor:Delta} shows that
\begin{align*}
&(\log\Ztilde(\Psi\circ \gamma))_{\leq 2} \\
&= \sum_{i=1}^{g} \Biggl( \strutgraph{i^+}{i^-} +\sum_{j\in L^t_i}\strutgraph{i^+}{j^-} \\
&+\sum_{k\in L^t_i}\biggl(-\frac{1}{2}\lambdagraph{0.4}{i^+}{i^-}{k^-} +\frac{1}{4}\Hgraph{i^+}{i^+}{i^-}{k^-} +\frac{1}{12}\lamlamref{i^+}{i^-}{k^-}{k^-} +\frac{1}{12}\lamlam{i^+}{i^-}{i^-}{k^-}\biggr) \\
&\hspace{-2em} +\sum_{j\prec k\in L^t_i}\biggl(-\frac{1}{2}\lambdagraph{0.4}{i^+}{j^-}{k^-} +\frac{1}{4}\Hgraph{i^+}{i^+}{j^-}{k^-} +\frac{1}{12}\lamlamref{i^+}{j^-}{k^-}{k^-} +\frac{1}{12}\lamlam{i^+}{j^-}{j^-}{k^-} +\frac{1}{6}\lamlamref{i^+}{i^-}{j^-}{k^-} +\frac{1}{6}\lamlam{i^+}{i^-}{j^-}{k^-} \biggr) \\
&+\sum_{j\prec k\prec l\in L^t_i}\biggl(\frac{1}{6}\lamlamref{i^+}{j^-}{k^-}{l^-} +\frac{1}{6}\lamlam{i^+}{j^-}{k^-}{l^-}\biggr)
+\sum_{j\in L^b_i,\, k\in L^t_i}\frac{1}{4}\lamlamref{i^+}{i^-}{j^-}{k^-} \\
&+\sum_{j\in L^b_i}\biggl(\strutbottom{i^-}{j^-} -\frac{1}{2}\lambdagraph{0.4}{i^+}{i^-}{j^-} +\frac{1}{12}\Hgraph{i^+}{i^+}{i^-}{j^-} +\frac{1}{12}\lamlamref{i^+}{i^-}{j^-}{j^-} -\frac{1}{8}\Hbottom{i^-}{i^-}{j^-}{j^-} +\frac{1}{8}\dumbbell{i^-}{j^-}\biggr) \\
&+\sum_{j\prec k\in L^b_i}\biggl(-\frac{1}{2}\Ybottom{i^-}{j^-}{k^-} +\frac{1}{6}\lamlam{i^+}{i^-}{j^-}{k^-} +\frac{1}{6}\lamlamref{i^+}{i^-}{j^-}{k^-} +\frac{1}{12}\comb{i^-}{j^-}{j^-}{k^-} +\frac{1}{12}\Hbottom{i^-}{j^-}{k^-}{k^-} \biggr) \\
&+\sum_{j\prec k\prec l\in L^b_i}\biggl(\frac{1}{6}\comb{i^-}{j^-}{k^-}{l^-} +\frac{1}{6}\Hbottom{i^-}{j^-}{k^-}{l^-}\biggr) \Biggr)
+\sum_{(j,k)\in E}\biggl(-\strutbottom{j^-}{k^-}+\frac{1}{8}\dumbbell{j^-}{k^-}\biggr).
\end{align*}
Using this result, we now compute $(-1)^{\# E}\left((\Id_g\otimes(\emptyset - \Ztilde(Y))^{\otimes n}_{n})\circ \Ztilde_{2}(\Psi\circ \gamma)\right)^Y$, where the superscript $Y$ denotes the projection appearing in Section~\ref{subsec:Jacobi}.
Since
\[
(\Id_g\otimes(\emptyset - \Ztilde(Y))^{\otimes n}_{n})\circ \Ztilde_{2}(\Psi\circ \gamma)
= (\Id_g\otimes \Ztilde_1(Y)^{\otimes n})\circ \Ztilde_{2}(\Psi\circ \gamma)
\]
and $\Ztilde_1(Y)^{\otimes n}$ does not have repeated labels,
it suffices to consider Jacobi diagrams in $\Ztilde_{2}(\Psi\circ \gamma)$ which do not have the same labels in $\{j^-\mid j\in N\}$.
Therefore, in $\A_{n+2}^Y\otimes\Q/\tfrac{1}{2}\Z$, the above value is equal to
\begin{align*}
&\sum_{\{u,v\}} \frac{1}{4}\delta^{Y}_{u}(\delta^{Y}_{v}(J)) 
+\sum_{v\in U^+} \biggl(\frac{1}{4}\delta^{+}_{v}(J) +\frac{1}{12}\delta^{-}_{v}(J)\biggr) 
+\sum_{\substack{\{\{u,v\},\{u',v'\}\}\\ u\prec v,\, u'\prec v'}} \frac{1}{4}\lambda_{u,v}(\lambda_{u',v'}(J)) 
\\
&+\sum_{\substack{u,v,w\in U\\ u\prec v}} \frac{1}{4}\lambda_{u,v}(\delta^{Y}_{w}(J)) 
+\sum_{u\prec v\in U^+} \biggl(\frac{1}{4}H_{u,v}(J) +\frac{1}{6}H'_{u,v}(J)\biggr)
+\sum_{u\prec v\prec w\in U^+} \frac{1}{6}\lambda_{u,v,w}(J) \\
&+\sum_{\substack{\{u,v\}\\ \ell(u)=\ell(v)^\ast}} \frac{1}{4}H_{u,v}(J) 
+\sum_{v\in U^-}\frac{1}{12}\delta^{+}_{v}(J) 
+\sum_{v\in U^-}\frac{1}{8}\beta_{e(v)}(J) 
+\sum_{u\prec v\in U^-}\frac{1}{6}H'_{v,u}(J) \\
&+\sum_{u\prec v\prec w\in U^-} \frac{1}{6}\lambda_{u,v,w}(J) 
+\sum_{\text{$e$: internal edge}}\frac{-1}{8}\beta_e(J). 
\end{align*}
Here, $\frac{1}{4}\delta^{Y}_{u}(\delta^{Y}_{v}(J))$ is obtained from two $-\frac{1}{2}\lambdagraph{0.4}{i^+}{i^-}{k^-}$, 
$\frac{1}{4}\delta^{+}_{v}(J)$ from $\frac{1}{4}\Hgraph{i^+}{i^+}{i^-}{k^-}$, 
$\frac{1}{12}\delta^{-}_{v}(J)$ from $\frac{1}{12}\lamlam{i^+}{i^-}{i^-}{k^-}$, 
$\frac{1}{4}\lambda_{u,v}(\lambda_{u',v'}(J))$ from two $-\frac{1}{2}\lambdagraph{0.4}{i^\pm}{j^-}{k^-}$, 
$\frac{1}{4}\lambda_{u,v}(\delta^{Y}_{w}(J))$ from $-\frac{1}{2}\lambdagraph{0.4}{i^\pm}{j^-}{k^-}$ and $-\frac{1}{2}\lambdagraph{0.4}{i^+}{i^-}{k^-}$, 
$\frac{1}{4}H_{u,v}(J)$ from $\frac{1}{4}\Hgraph{i^\pm}{i^+}{j^-}{k^-}$, 
$\frac{1}{6}H'_{u,v}(J)$ from $\frac{1}{4}\lamlamref{i^+}{i^-}{j^-}{k^-}$ and four other similar terms, 
$\frac{1}{6}\lambda_{u,v,w}(J)$ from $\frac{1}{6}\lamlamref{i^\pm}{j^-}{k^-}{l^-} +\frac{1}{6}\lamlam{i^\pm}{j^-}{k^-}{l^-}$, 
$\frac{1}{8}\beta_{e(v)}(J)$ from $\frac{1}{8}\dumbbell{i^-}{j^-}$, 
$\frac{-1}{8}\beta_e(J)$ from $\frac{1}{8}\dumbbell{j^-}{k^-}$.

Similarly, we compute $(-1)^{\# E}\left((\Id_g\otimes(\emptyset - \Ztilde(Y))^{\otimes n}_{n+1})\circ \Ztilde_{1}(\Psi\circ \gamma)\right)^Y$:
\begin{align*}
&\sum_{Y} \frac{1}{4}\delta^Y(J\sqcup Y)
+\sum_{Y} \frac{1}{4}\lambda(J\sqcup Y) \\
&+\sum_{u\in U(J),\, v\in U(\delta_u^{\shortmid\shortmid}(J))} \frac{1}{4}\delta_v^Y(\delta_u^{\shortmid\shortmid}(J))
+\sum_{\substack{u\in U(J),\, v,w\in U(\delta_u(J))\\ U(J)\cap \{v,w\}\neq \emptyset}} \frac{1}{4}\lambda_{v,w}(\delta^{\shortmid\shortmid}_{u}(J)).
\end{align*}

Finally, we compute $(-1)^{\# E}\left((\Id_g\otimes(\emptyset - \Ztilde(Y))^{\otimes n}_{n+2})\circ \Ztilde_{0}(\Psi\circ \gamma)\right)^Y$.
Note that $\Ztilde_{\leq 3}(Y)$ is given by $\exp_\sqcup$ of diagrams in Lemma~\ref{lem:Ycob}.
Let
\[
\delta^{YY}_{u,v}\Bigl(
\begin{tikzpicture}[scale=0.3, baseline={(0,0.4)}, densely dashed]
  \draw (1,0) -- (1,2);
  \draw (1,2) -- (0,3) node[anchor=south] {\Small $\ell(u)$};
  \draw (1,2) -- (2,3) node[anchor=south] {\Small $\ell(v)$};
  \draw (0,0) -- (2,0);
\end{tikzpicture}\Bigr)
=
\begin{tikzpicture}[scale=0.3, baseline={(0,0.4)}, densely dashed]
  \draw (1,0) -- (1,2);
  \draw (1,2) -- (0,3) node[anchor=south] {\Small $\ell(u)$};
  \draw (1,2) -- (2,3) node[anchor=south] {\Small $\ell(v)$};
  \draw (5,0) -- (5,2);
  \draw (5,2) -- (4,3) node[anchor=south] {\Small $\ell(u)$};
  \draw (5,2) -- (6,3) node[anchor=south] {\Small $\ell(v)$};
  \draw (0,0) -- (6,0);
\end{tikzpicture},
\quad
\delta^{\shortmid\shortmid\shortmid'}_{u,v}\Bigl(\Ygraph{0.4}{\ell(u)}{\ell(v)}{}\Bigr)
=
\begin{tikzpicture}[scale=0.4, baseline={(0,0)}, densely dashed]
  \draw (0,-1) -- (0,1) node[anchor=south] {\Small $\ell(u)$};
  \draw (2,0) -- (2,1) node[anchor=south] {\Small $\ell(v)$};
  \draw (4,0) -- (4,1) node[anchor=south] {\Small $\ell(u)$};
  \draw (0,0) -- (5,0) node[anchor=north] {\Small $\ell(v)$};
\end{tikzpicture}.
\]
Then, in $\A_{n+2}^Y\otimes\Q/\tfrac{1}{2}\Z$, $(-1)^{\# E}\left((\Id_g\otimes(\emptyset - \Ztilde(Y))^{\otimes n}_{n+2})\circ \Ztilde_{0}(\Psi\circ \gamma)\right)^Y$ equals
\begin{align*}
&\sum_{Y}\frac{1}{3!}J\sqcup Y^{\sqcup 2}
+\sum_{\text{$\{u,v\}$: leaf pair}} \frac{-1}{2}\delta^{YY}_{u,v}(J)
+\sum_{Y}\frac{1}{4}\delta^{\shortmid\shortmid}(J\sqcup Y)
+\sum_{\text{$\{u,v\}$: non-leaf pair}} \frac{1}{4}\delta^{\shortmid\shortmid}_{u}(\delta^{\shortmid\shortmid}_{v}(J)) 
\\
&+\sum_{v\in U}\frac{1}{6}\delta^{\shortmid\shortmid\shortmid}_{v}(J) 
+\sum_{\text{$\{u,v\}$: leaf pair}}\frac{-1}{4}\delta^{\shortmid\shortmid\shortmid'}_{u,v}(J)
+\sum_{\text{$e$: internal edge}}\frac{1}{8}\beta_e(J),
\end{align*}
where the second term is obtained by connecting two $-\Ytop{j^+}{k^+}{l^+}$ and $\frac{1}{2}\Htop{p^+}{p^+}{q^+}{r^+}$ with two $-\strutbottom{l^-}{p^-}$, and the last term is obtained by connecting $\frac{1}{2}\Htop{j^+}{j^+}{k^+}{l^+}$ and $\frac{1}{2}\Htop{p^+}{p^+}{q^+}{r^+}$ with two $-\strutbottom{j^-}{p^-}$ and by the AS and IHX relations.
Furthermore, it follows from the AS and IHX relations that
\[
\sum_{\text{$\{u,v\}$: non-leaf pair}} \frac{1}{4}\delta^{\shortmid\shortmid}_{u}(\delta^{\shortmid\shortmid}_{v}(J)) 
+\sum_{\text{$\{u,v\}$: leaf pair}} \frac{-1}{4}\delta^{\shortmid\shortmid\shortmid'}_{u,v}(J)
=
\sum_{\{u,v\}} \frac{1}{4}\delta^{\shortmid\shortmid}_{u}(\delta^{\shortmid\shortmid}_{v}(J)).
\]
Combining the three computations above, we obtain the desired formula.
\end{proof}

\section{Computation of the group $Y_n\I\C/Y_{n+1}$}
\label{sec:Yn/Yn+1}
In this section, we investigate the abelian group $Y_n\I\C/Y_{n+1}$ for $n=5,6,7$.
More precisely, we give the proofs of Theorems~\ref{thm:Y6/Y7} and \ref{thm:deg-by-deg} in Sections~\ref{subsec:Y6/Y7} and \ref{subsec:Y7/Y8}, respectively.

\subsection{Computation of $Y_5\I\C/Y_6$}
\label{subsec:Y5/Y6}
This subsection is devoted to giving an upper bound and lower bound of the size of $\tor(Y_5\I\C/Y_6)$.

\begin{proposition}
Let $g$ be a non-negative integer.
Then, the abelian group $\tor(Y_5\I\C/Y_6)$ is isomorphic to $(\Z/2\Z)^r$ for some $r$ satisfying
\[
4g^3+6g^2 \leq r \leq 4g\binom{2g+1}{3}+4g^3+6g^2.
\]
\end{proposition}

\begin{proof}
Since the homomorphism $\ss_5\colon \A^c_5 \to Y_5\I\C/Y_6$ is surjective and an isomorphism over $\Q$, we have $(\tor\A^c_5)/\Ker\ss_5 \cong \tor(Y_5\I\C/Y_6)$.
To investigate the left-hand side, we investigate the abelian group $\A^c_5=\bigoplus_{l=0}^{3} \A^c_{5,l}$.
We first have $\tor\A^c_{5,0} \cong (H\otimes L_3)\otimes\Z/2\Z$ by \cite[Corollary~1.2]{CST12L} whose rank is $2g\frac{1}{3}((2g)^3-2g) = 4g\binom{2g+1}{3}$ by Witt's formula for (see \cite[Theorem~5.11]{MKS04} for example).
Next, \cite[Proposition~5.2]{NSS22GT} shows $\tor\A^c_{5,1} \cong H^{\otimes 3}\otimes\Z/2\Z$ and \cite[Lemma~4.4]{NSS22JT} implies $\tor\A^c_{5,2} \cong \tor\A^c_{1,0} \cong H^{\otimes 2}\otimes\Z/2\Z$.
Finally, we have $\A^c_{5,3}=0$ by \cite[Lemma~5.30]{CDM12}.
Thus, $\rank(\tor\A^c_5) = 4g\binom{2g+1}{3}+(2g)^3+(2g)^2$.

Let us give the upper bound.
We have 
\[
\Ker\ss_{5,1} = \Ker(\pi\circ\ss_{5,1}) \cong (\Z/2\Z)^{4g^3-2g^2},
\]
where the second isomorphism is in \cite[Theorem~1.1]{NSS22JT}, and the first equality comes from \cite[Remark~3.18]{NSS22JT}.
Hence, $r\leq \rank(\tor\A^c_5)-(4g^3-2g^2)$ as desired.
To give the lower bound, we estimate the size of the image of $\bar{z}_6\colon \tor(Y_5\I\C/Y_6)\to \A^c_6\otimes\Q/\Z$.
It follows from the proof of \cite[Theorem~1.1]{NSS22JT} that elements $\bar{z}_{6,1}(\ss_6(O(a,b,c,b,a))) = \frac{1}{2}O(a,b,c,c,b,a)$ generate a submodule of rank $4g^3+2g^2$.
\cite[Theorem~1.1]{NSS22GT} shows $\bar{z}_{6,3}(\ss_6(\theta(a,a;;b))) = \bu^{(2)}(O(a,b))$ and these elements generate a submodule of rank $(2g)^2$ by the proof of Proposition~\ref{prop:A6_free} in the next subsection, where $\bu^{(2)}$ is a map $\A^c_{2,1}\to \A^c_{6,3}$ defined in \cite[Definition~4.1]{NSS22JT}.
Therefore, $r\geq 4g^3+2g^2+(2g)^2$.
\end{proof}

\begin{remark}
To determine the above $r$ exactly, we would need to investigate $\Ker\ss_{5,0}$.
\end{remark}

\subsection{Computation of $Y_6\I\C/Y_7$}
\label{subsec:Y6/Y7}
Here, we use Theorem~\ref{thm:formula} to prove Theorem~\ref{thm:Y6/Y7} which asserts that $\tor(Y_6\I\C/Y_7)$ is generated by torsion elements of order $3$.

Recall that clasper surgery induces an exact sequence
\[
0\to \Ker\ss_6 \to \A^c_{6}\xrightarrow{\ss_6} Y_6\I\C/Y_7 \to 0.
\]
We compute the composite map
\[
\A^c_{6,2} \xrightarrow{\ss_6} Y_6\I\C/Y_7 \xrightarrow{\zz_{8,4}} \A^c_{8,4}\otimes\Q/\tfrac{1}{2}\Z.
\]

Let $\A^c_{n,l}(a_1,\dots,a_m)$ denote the submodule of $\A^c_{n,l}$ generated by Jacobi diagrams whose labels are precisely $a_1,\dots,a_m$.
For instance, $\A^c_{1,0}(a,a,b)$ is isomorphic to $\Z/2\Z$ generated by $T(a,a,b)$ for $a,b\in\{1^\pm,\ldots,g^\pm\}$.
We recall from \cite[Section~4.1]{NSS22JT} that the \emph{spine} of a Jacobi diagram $J$ is defined to be the graph obtained by collapsing edges incident to univalent vertices until there is no univalent vertex.

\begin{lemma}
\label{lem:spine_simple}
When $l\geq 3$, the module $\A^c_{n,l}$ is generated by Jacobi diagrams whose spines are simple graphs.
\end{lemma}

\begin{proof}
First note that a graph is said to be \emph{simple} if it contains no self-loop and no multiple edge.
Let $J$ be a Jacobi diagram whose spine contains self-loops.
Here, the assumption implies that the spine is a connected trivalent graph with at least four vertices.
For any self-loop, let $e$ be the edge (in the spine) connecting the loop and the rest.
All edges (of $J$) attached to $e$ can be moved to the rest by the IHX relation.
Then, we eliminate the loop by applying the IHX relation to $e$.
Applying this process to every self-loop, we express $J$ as a linear combination (over $\Z$) of Jacobi diagrams $J'$ whose spines have no self-loops.

Now, the spine of $J'$ could have multiple edges.
Let $e_1$ and $e_2$ be multiple edges connecting vertices $u$ and $v$.
Let $u'$ (resp.\ $v'$) be the vertex adjacent to $u$ (resp.\ $v$) different from $v$ (resp.\ $u$).
In the case of $u'\neq v'$, one can eliminate the multiple edges by the IHX relation for $(u,u')$ without creating new multiple edges and self-loops.
In the case of $u'= v'$, using the IHX relation twice, we eliminate the multiple edges as follows:
\[
\begin{tikzpicture}[scale=0.6, baseline={(0,0.3)}, densely dashed]
 \draw (0,0) circle [radius=1];
 \draw (-1,0) -- (1,0) node[at start, anchor=east] {\Small $u$} node[at end, anchor=west] {\Small $v$};
 \draw (0,2) -- (0,1) node[anchor=south west] {\Small $u'=v'$};
 \draw (-1,2) -- (1,2);
 \draw[dotted] (-1,2) -- (-2,2);
 \draw[dotted] (1,2) -- (2,2);
 \node at (0,0.3) {\Small $e_1$};
 \node at (0,-0.7) {\Small $e_2$};
\end{tikzpicture}%
\quad \leadsto \quad
\begin{tikzpicture}[scale=0.6, baseline={(0,-0.1)}, densely dashed]
 \draw (0,0) circle [radius=1];
 \draw (0,-1) -- (0,1);
 \draw (1,0) -- (2,0);
 \draw (-1,0) -- (-2,0);
 \draw[dotted] (-2,0) -- (-3,0);
 \draw[dotted] (2,0) -- (3,0);
\end{tikzpicture}\ .
\]
This completes the proof.
\end{proof}

\begin{proposition}
\label{prop:A6_free}
$\A^c_{6,l}$ is a free $\Z$-module unless $l=2$.
\end{proposition}

\begin{proof}
We consider $l=0,1,3,4$ since $\A^c_{6,l}$ is trivial for $l\geq 5$.
The cases $l=0,1$ follow from \cite[Corollary~1.2]{CST12L} and \cite[Proposition~5.2]{NSS22GT}, respectively.
Next, we consider $\A^c_{6,4}$ which is a module generated by Jacobi diagrams with no univalent vertex.
By Lemma~\ref{lem:spine_simple}, it suffices to consider simple trivalent graphs with $6$ vertices, which are either the $1$-skeleton of a triangular prism $\bu^{(2)}(\theta)$ or the complete bipartite graph $K_{3,3}$, where $\theta$ denotes the theta graph.
The latter is changed into the former by the IHX relation, and thus $\A^c_{6,4}$ is generated by $\bu^{(2)}(\theta)$.
Here $\bu^{(2)}(\theta)$ is of infinite order since $W_{\sl_2(\mathbb{C})}(\bu^{(2)}(\theta))=-6$, where $W_{\sl_2(\mathbb{C})}$ is the weight system associated with the Lie algebra $\sl_2(\mathbb{C})$.
For details on weight systems, we refer the reader to \cite[Definition~6.2 and Example~6.3]{NSS22JT} or \cite[Section~6.3]{CDM12}.

Finally, we discuss $\A^c_{6,3}$.
By Lemma~\ref{lem:spine_simple}, $\A^c_{6,3}(a,b)$ is generated by $\bu(\theta)$ attached with two hairs whose vertices are colored with $a$ and $b$, respectively, where a hair is an edge incident to one univalent vertex.
Moreover, the IHX relation implies that $\A^c_{6,3}(a,b)$ is generated by $\bu^{(2)}(O(a,b))$.
Here, it is of infinite order since
\[
W_{\sl_2(\mathbb{C})}(\bu^{(2)}(O(a,b)))=-2\sum_{i=1}^3(a\otimes e_i)(b\otimes e_i)
\]
is non-trivial.
\end{proof}

Recall the notation $\theta(a_1,\dots,a_p;b_1,\dots,b_q;c_1,\dots,c_r)$ introduced in Theorem~\ref{thm:deg-by-deg}.

\begin{proposition}
\label{prop:A6_torsion}
For $a,b\in\{1^\pm,\ldots,g^\pm\}$, $\A^c_{6,2}(a,a,a,b)$ is isomorphic to $\Z/3\Z\oplus \Z$ and generated by $\theta(a,b;a;a)$ and $\theta(a,b;;a,a)$.
\end{proposition}

\begin{proof}
By \cite[Proposition~4.2]{NSS22JT}, it suffices to consider the theta graph.
Let us first discuss the case $a=b$.
Under the AS relation, every Jacobi diagram in $\A^c_{6,2}(a,a,a,a)$ is equivalent to one of $\theta(a,a,a,a;;)$, $\theta(a,a,a;a;)$, $\theta(a,a;a,a;)$, or $\theta(a,a;a;a)$.
Considering all the relations among these four elements coming from the IHX (and AS) relations such as $\theta(a,a,a,a;;)+2\theta(a,a,a;a;)=0$,
we obtain a presentation of the module $\A^c_{6,2}(a,a,a,a)$ and its Smith normal form:
\[
\begin{pmatrix}
1&0&0\\
2&1&0\\
0&1&0\\
0&-1&3
\end{pmatrix}
\leadsto
\begin{pmatrix}
1&0&0\\
0&1&0\\
0&0&3\\
0&0&0
\end{pmatrix}.
\]
This implies that $\A^c_{6,2}(a,a,a,a)\cong \Z/3\Z \oplus \Z$, which is generated by $\theta(a,a;a;a)$ and $\theta(a,a;;a,a)=\theta(a,a;a,a;)$.

Next, let us discuss the case $a\neq b$.
One can check that $\A^c_{6,2}(a,a,a,b)$ is generated by $\theta(a,b;a;a)$ and $\theta(a,b;;a,a)$.
On the other hand, we have a surjective homomorphism $\A^c_{6,2}(a,a,a,b)\twoheadrightarrow \A^c_{6,2}(a,a,a,a)$ defined by replacing $b$ with $a$.
Since $3\theta(a,b;a;a)=0$ by the AS and IHX relations, the homomorphism must be an isomorphism.
\end{proof}

\begin{remark}
\label{rem:Ishikawa}
Katsumi Ishikawa informed the first author about the existence of torsion elements rather than $2$-torsions, and the above explicit elements were found by the authors.
In particular, he announced that $\tor\A^c_{6,2}(a,a,a,a)\cong \Z/3\Z$.
\end{remark}

\begin{remark}
More generally, for $a_1,\dots,a_k,b \in \{1^\pm,\dots,g^\pm\}$, it holds that
\[
3\theta(a_1,\dots,a_k,b; a_1,\dots,a_k; a_1,\dots,a_k) =0 \in \A^c_{3k+3,2}
\]
by the AS and IHX relations.
\end{remark}

\begin{remark}
By \cite[Theorem~1.3]{NSS22JT}, we have an isomorphism
\[
\bu\colon \A^c_{4,1}\to \A^c_{6,2}/\ang{\Theta_6^{\geq 1}}.
\]
Here recall from \cite[Proposition~5.2]{NSS22GT} that $\A^c_{4,1}(a,a,a,b)\cong\Z$.
As a corollary of Proposition~\ref{prop:A6_torsion}, $\bu$ induces
\[
\A^c_{4,1}(a,a,a,b)\cong \A^c_{6,2}(a,a,a,b)/{\tor}.
\]
\end{remark}

By a computer-aided calculation, we can obtain a presentation of the module $\A^c_{6,2}(a,b,c,d)$ and its Smith normal form in much the same way as the proof of Proposition~\ref{prop:A6_torsion}.
As a consequence, we obtain the following.

\begin{proposition}
\label{prop:A6_rest}
Suppose any three of $a,b,c,d\in\{1^\pm,\ldots,g^\pm\}$ are not the same.
Then $\A^c_{6,2}(a,b,c,d)$ is a free $\Z$-module.
\end{proposition}

\begin{figure}[h]
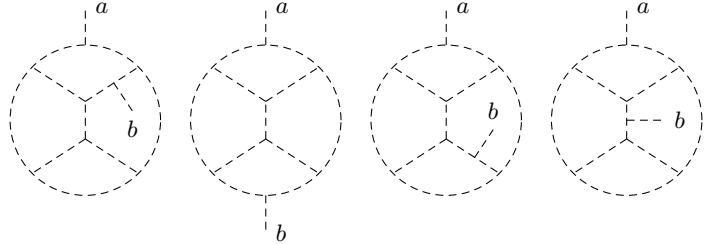

 \centering
$
\Pone{a}{b}
\quad
\Ptwo{a}{b}
\quad
\Pthree{a}{b}
\quad
\Pfour{a}{b}
$
 \caption{Four Jacobi diagrams denoted by $P_k(a,b)$ ($k=1,2,3,4$).}
 \label{fig:A84}
\end{figure}

\begin{proposition}
\label{prop:A84}
For $a,b\in\{1^\pm,\ldots,g^\pm\}$, $\A^c_{8,4}(a,b)$ is a free abelian group with basis $\{P_1(a,b),P_2(a,b)\}$ \textup{(}see Figure~\ref{fig:A84}\textup{)}.
\end{proposition}

\begin{proof}
In the same way as the proof of Proposition~\ref{prop:A6_free}, we see that $\A^c_{8,4}$ is generated by $\bu^{(2)}(\theta)$ attached with two hairs, that is, the Jacobi diagrams listed in Figure~\ref{fig:A84}.
One can see that 
\[
-P_1(a,b)+P_2(a,b)=P_3(a,b),\ P_4(a,b)=0,\text{ and } P_k(a,b)=P_k(b,a)
\]
for $k=1,2,3$, and hence $\A^c_{8,4}(a,b)$ is generated by $P_1(a,b)$ and $P_2(a,b)$.
On the other hand, we have a homomorphism $\A^c_{8,4}(a,b) \to \A^c_{8,5}$ by gluing two univalent vertices.
According to \cite[Table~7.1]{CDM12}, this map induces an isomorphism over $\Q$.
Therefore, $P_1(a,b)$ and $P_2(a,b)$ are linearly independent over $\Z$.
\end{proof}

\newcommand{\lambdaex}{
\begin{tikzpicture}[scale=0.4, baseline={(0,0)}, densely dashed]
 \draw (0,0) circle [radius=2];
 \draw (-2,0)--(2,0);
 \draw (135:2) .. controls +(-1,1) and +(0,-1) .. (-1,4);
 \draw (0,0)--(0,4);
 \draw (0,-2) .. controls +(0.4,0.4) and +(2,0) .. (0,4);
 \draw (0,4)--(-2,4) node[anchor=east] {\Small $a$};
 \draw (30:2)--(30:3) node[anchor=south] {\Small $b$};
\end{tikzpicture}}
\newcommand{\lambdarefex}{
\begin{tikzpicture}[scale=0.4, baseline={(0,0)}, densely dashed]
 \draw (0,0) circle [radius=2];
 \draw (-2,0)--(2,0);
 \draw (0,-2) .. controls +(1,1.7) and +(0,-1) .. (1,4);
 \draw (0,0)--(0,4);
 \draw (135:2) .. controls +(-1,1) and +(-2,0) .. (0,4);
 \draw (0,4)--(2,4) node[anchor=west] {\Small $a$};
 \draw (30:2)--(30:3) node[anchor=south] {\Small $b$};
\end{tikzpicture}}
\newcommand{\Pfive}[2]{
\begin{tikzpicture}[scale=0.25, baseline={(0,-0.2)}, densely dashed]
 \draw (0,0) circle [radius=4];
 \draw (0,-1)--(0,1);
 \draw (0,1)--(45:4);
 \draw (0,1)--(135:4);
 \draw (0,-1)--(-45:4);
 \draw (0,-1)--(-135:4);
 \draw (4,0)--(6,0) node[anchor=south] {\Small $#2$};
 \draw (0,4)--(0,6) node[anchor=west] {\Small $#1$};
\end{tikzpicture}}

\begin{proof}[Proof of Theorem~\ref{thm:Y6/Y7}]
We first recall that $\ss\colon \A^c_6 \to Y_6\I\C/Y_7$ is surjective and induces an isomorphism over $\Q$.
It follows from Propositions~\ref{prop:A6_free}, \ref{prop:A6_torsion}, and \ref{prop:A6_rest} that $\tor(Y_6\I\C/Y_7)$ is isomorphic to $(\Z/3\Z)^r$ for some $r$ satisfying
\[
r\leq \rank_{\Z/3\Z}(\tor\A^c_{6}) = (2g)^2 = 4g^2.
\]
Let us show $r\geq \binom{2g}{2}$ by the map
\[
\zz_{8,4}\circ\ss_6\colon \A^c_{6,2}\to \A^c_{8,4}\otimes\Q/\tfrac{1}{2}\Z.
\]
Since $\A^c_{8,4} = \bigoplus_{a,b}\A^c_{8,4}(a,b)$, if $(\zz_{8,4}\circ\ss_6)(\theta(a,b;a;a))\neq 0$ is shown for distinct $a,b\in \{1^\pm,\dots,g^\pm\}$, then we conclude that $r\geq \binom{2g}{2}$.
By Theorem~\ref{thm:formula}, we have
\begin{align*}
& (\zz_{8,4}\circ\ss_6)(\theta(a,b;a;a))
= (-1)^{2+1}\sum_{u\prec v\prec w\in U} \frac{1}{6}\lambda_{u,v,w}(\theta(a,b;a;a)) \\
&= \frac{1}{6}\biggl(\lambdaex +\ \lambdarefex\biggr)
\\
&= \frac{1}{6}\biggl(
-
\begin{tikzpicture}[scale=0.4, baseline={(0,0)}, densely dashed]
 \draw (0,0) circle [radius=2];
 \draw (-2,0)--(2,0);
 \draw (135:2) .. controls +(-1,1) and +(0,-1) .. (-1,4);
 \draw (0,0)--(225:2);
 \draw (0,-2) .. controls +(0.4,0.4) and +(2,0) .. (0,4);
 \draw (0,4)--(-2,4) node[anchor=east] {\Small $a$};
 \draw (30:2)--(30:3) node[anchor=south] {\Small $b$};
\end{tikzpicture}
+
\begin{tikzpicture}[scale=0.4, baseline={(0,0)}, densely dashed]
 \draw (0,0) circle [radius=2];
 \draw (-2,0)--(2,0);
 \draw (135:2) .. controls +(-1,1) and +(0,-1) .. (-1,4);
 \draw (0,0)--(-45:2);
 \draw (0,-2) .. controls +(0.4,0.4) and +(2,0) .. (0,4);
 \draw (0,4)--(-2,4) node[anchor=east] {\Small $a$};
 \draw (30:2)--(30:3) node[anchor=south] {\Small $b$};
\end{tikzpicture}
-
\begin{tikzpicture}[scale=0.4, baseline={(0,0)}, densely dashed]
 \draw (0,0) circle [radius=2];
 \draw (-2,0)--(2,0);
 \draw (0,-2) .. controls +(1,1.7) and +(0,-1) .. (1,4);
 \draw (0,0)--(0,2);
 \draw (135:2) .. controls +(-1,1) and +(-2,0) .. (0,4);
 \draw (0,4)--(2,4) node[anchor=west] {\Small $a$};
 \draw (30:2)--(30:3) node[anchor=south] {\Small $b$};
\end{tikzpicture}
+\ 
\begin{tikzpicture}[scale=0.4, baseline={(0,0)}, densely dashed]
 \draw (0,0) circle [radius=2];
 \draw (-2,0)--(2,0);
 \draw (0,-2) .. controls +(1,1.7) and +(0,-1) .. (1,4);
 \draw (0,0)--(150:2);
 \draw (120:2) .. controls +(-1,1) and +(-2,0) .. (0,4);
 \draw (0,4)--(2,4) node[anchor=west] {\Small $a$};
 \draw (30:2)--(30:3) node[anchor=south] {\Small $b$};
\end{tikzpicture}
\biggr).
\end{align*}
The first term cancels with the fourth term, and the other two terms are equal to
\begin{align*}
\Pfive{a}{b}-\ \Ptwo{a}{b}
\ = -P_3(a,b)
= P_1(a,b)-P_2(a,b).
\end{align*}
Here, Proposition~\ref{prop:A84} implies that $\frac{1}{6}(P_1(a,b)-P_2(a,b))\neq 0$ in $\A^c_{8,4}\otimes\Q/\tfrac{1}{2}\Z$.
This completes the proof.
\end{proof}

\begin{remark}
The authors do not know whether $\theta(a,b;a;a)-\theta(b,a;b;b) \in \Ker\ss_6$ or not.
\end{remark}

\subsection{Computation of $Y_7\I\C/Y_8$ and $\Ker\ss_{n,l}$}
\label{subsec:Y7/Y8}
Let us prove that the inclusion $\bigoplus_{l\geq 0}\Ker\ss_{7,l} \subset \Ker\ss_{7}$ is strict.
A key of the proof is a homomorphism $\bar{z}_{8}\colon Y_7\I\C/Y_8 \to \A^c_8\otimes\Q/\Z$.

\begin{lemma}
\label{lem:A_8_2}
For distinct $a,b\in \{1^\pm,\dots,g^\pm\}$, the diagram $\theta(a;a,a;a,b,a)$ is a primitive element in $\A^c_{8,2}$.
\end{lemma}

\begin{proof}
By the AS and IHX relations, each Jacobi diagram in $\A^c_{8,2}(a,a,a,a,a,b)$ is expressed as a linear combination of diagrams of the form $\theta(\ast;\ast;a,b,a)$.
Therefore, $\A^c_{8,2}(a,a,a,a,a,b)$ is generated by $\theta(a;a,a;a,b,a)$ and $\theta(a,a,a;;a,b,a)$.
Moreover, by \cite[Proposition~4.2]{NSS22JT} and a computer program, we check that the two elements form a basis over $\Z$.
\end{proof}

\begin{proof}[Proof of Theorem~\ref{thm:deg-by-deg}]
It follows from \cite[Corollary~3.17]{NSS22JT} that the sum
\[
O(a,a,a,b,a,a,a)+O(b,a,a,a,a,a,b)+\theta(a;a;a,b,a)+\theta(a,a,a;a;b)
\]
lies in $\Ker\ss_7$.
Hence, it suffices to see that
\[
O(a,a,a,b,a,a,a)+O(b,a,a,a,a,a,b) \notin \Ker\ss_7.
\]
Its image under the map
\[
Y_7\I\C/Y_8 \xrightarrow{\bar{z}_{8,2}} \A^c_{8,2}\otimes\Q/\Z \xrightarrow{\pr} \A^c_{8,2}(a,a,a,a,a,b)\otimes\Q/\Z
\]
is equal to
\begin{align}
\frac{1}{2}\theta(a;;a,a,b,a,a) +\frac{1}{2}\theta(a,a;a;a,b,a) +\frac{1}{2}\theta(a,a,a,a;a;b) +\frac{1}{2}\theta(a,a,a,a,a;;b)
\label{eq:aaaaab}
\end{align}
by \cite[Theorem~1.1]{NSS22GT}.
The sum of the first two terms equals $\frac{1}{2}\theta(a,a,a;;a,b,a)$ by the AS and IHX relations.
In a similar way, we see that \eqref{eq:aaaaab} is equal to $\frac{1}{2}\theta(a;a,a;a,b,a)$.
Thus, Lemma~\ref{lem:A_8_2} completes the proof.
\end{proof}

\begin{remark}
The proof answers negatively to the question in \cite[Remark~3.18]{NSS22JT}.
In much the same way, for $g\geq 2$ and distinct colors $a_1,a_2,a_3,a_4 \in \{1^\pm,\dots,g^\pm\}$, we can show that 
\[
\hspace{-2em}
O(a_1,a_2,a_3,a_4,a_3,a_2,a_1)+O(a_4,a_3,a_2,a_1,a_2,a_3,a_4)+\theta(a_1;a_2;a_3,a_4,a_3)+\theta(a_2,a_1,a_2;a_3;a_4)
\]
lies in the gap of $\bigoplus_{l\geq 0}\Ker\ss_{7,l} \subset \Ker\ss_{7}$.
\end{remark}

\def\cprime{$'$} \def\cprime{$'$} \def\cprime{$'$}

\end{document}